\documentclass[11pt]{article}

\usepackage[T1]{fontenc}
\usepackage{lmodern}
\usepackage{amsmath,amssymb,amsthm,mathtools}
\usepackage{booktabs,array}
\usepackage[margin=1.12in]{geometry}
\usepackage{microtype}
\usepackage[hidelinks,pdfusetitle]{hyperref}

\numberwithin{equation}{section}

\newtheorem{theorem}{Theorem}[section]
\newtheorem{lemma}[theorem]{Lemma}
\newtheorem{proposition}[theorem]{Proposition}
\newtheorem{corollary}[theorem]{Corollary}
\theoremstyle{remark}

\newcommand{\Z}{\mathbb Z}

\newcommand{\supp}{\operatorname{supp}}
\newcommand{\one}{\mathbf 1}
\newcommand{\cP}{\mathcal P}
\newcommand{\ip}[2]{\langle #1,#2\rangle}

\title{Positive Eigenvalues of Hadamard Square Roots\\ on Finite Abelian Groups}
\author{Wei Xie}
\date{7 October 2026}

\begin{document}

\maketitle

\begin{abstract}
For a finite abelian group $G$, let $\mu_2(G)$ be the minimum
number of positive eigenvalues of an entrywise positive real
symmetric group-circulant matrix whose entrywise square is
positive definite.  We prove that $\mu_2(G)\le5$ precisely
when $|G|\le15$, and that $\mu_2(G)=6$ precisely when $|G|=16$
or $G$ is cyclic of order $18$, $20$ or $22$.  For the
noncyclic groups of order at most fifteen, the minimum is
three for $C_2^2$ and five otherwise.
If the positive Fourier support of such a matrix is
$\{0,\pm r_1,\ldots,\pm r_d\}$, where $d\ge1$ and all $2d+1$
displayed elements are distinct, then there are integers
$\alpha_1,\ldots,\alpha_d$, not all zero, such that
$\sum_i\alpha_i r_i=0$ and
$\sum_i|\alpha_i|\le\max\{5,d+2\}$.
The bound five is sharp for $d=2$ and $d=3$.
\end{abstract}



\section{Introduction}

For a real symmetric matrix $A=(a_{ij})$ with positive entries, write
$A^{\circ t}=(a_{ij}^t)$ for its entrywise $t$th power, and let
$n_+(A)$ denote the number of its positive eigenvalues, counted with
multiplicity.
The Schur product theorem asserts that the entrywise product of two
positive semidefinite matrices is positive semidefinite.  The reverse
question for entrywise roots leads to an inertia problem: if $H$ has
positive entries and $H^{\circ2}$ is positive definite, how small can
$n_+(H)$ be?

For rank-one positive semidefinite Hermitian matrices with no zero or
negative real entries, Marcus and Sandy proved that the principal
entrywise square root has one more positive eigenvalue than negative
eigenvalues \cite[Theorem~7]{MarcusSandy}.
Reams gave a proof that a real symmetric matrix with zero diagonal,
positive off-diagonal entries and exactly one positive eigenvalue has
a nonsingular entrywise square root with exactly one positive eigenvalue
\cite[Theorem~2.9]{Reams}.  Garg and Aujla proved that, for a
real symmetric matrix $A$ with positive entries,
\[
 n_+(A)=1\quad\Longrightarrow\quad
 n_+(A^{\circ t})=1\qquad(0<t\le1).
\]
Under the same hypothesis, nonsingularity of $A$ also implies
nonsingularity of $A^{\circ t}$ for $0<t\le1$
\cite[Theorem~3.9]{GargAujla}.

FitzGerald and Horn proved that $A\mapsto A^{\circ t}$ preserves
positive semidefiniteness on entrywise nonnegative real symmetric
$n\times n$ matrices, for $n\ge2$ and $t>0$, exactly when $t$ is an integer
or $t\ge n-2$ \cite[Theorem~2.2]{FitzGeraldHorn}.
For transforms preserving full inertia, Belton,
Guillot, Khare, and Putinar proved a rigidity result: for any fixed
$k\ge0$, a function $f:\mathbb R\to\mathbb R$ preserves the inertia of
every real symmetric matrix with exactly $k$ negative eigenvalues, in
every dimension, if and only if $f(t)=ct$ for some $c>0$
\cite[Theorem~1.2]{BeltonEtAl}.

Translation invariance on a finite abelian group turns the question
into a finite Fourier problem.  Positive definiteness of such a kernel is equivalent to
positivity of all its Fourier coefficients, while entrywise squaring
becomes convolution of the original Fourier spectrum.  Fourier
characterizations of strictly positive definite functions on compact
abelian groups were studied by Emonds and F\"uhr \cite{EmondsFuehr}.
Fawzi, Saunderson, and Parrilo gave conditions for nonnegative functions
on finite abelian groups to admit sums-of-squares representations with
controlled Fourier support \cite{FawziSaundersonParrilo}.  In the
present problem, pointwise positivity of the inverse transform is coupled
with strict positivity of the spectral self-convolution, and the quantity
to be minimized is the number of positive spectral coefficients.

We write $B\succ0$ for positive definiteness of a matrix $B$ and
$f>0$ for pointwise positivity of a real-valued function $f$.

For a finite abelian group $G$, consider the matrices
$H_h=(h(x-y))_{x,y\in G}$, where $h$ is real, even and strictly
positive.  Let $\mu_2(G)$ be the minimum of $n_+(H_h)$ under
the additional condition $H_h^{\circ2}\succ0$.  For a cyclic
group $C_n$, write $\kappa_2(n)=\mu_2(C_n)$.

\begin{theorem}\label{thm:main}
Let $G$ be a finite abelian group.
\begin{enumerate}
\item One has
\[
 \mu_2(G)\le5\quad\Longleftrightarrow\quad |G|\le15.
\]
For noncyclic $G$ of order at most fifteen, the minimum is
three for $C_2^2$ and five otherwise.
\item The equality $\mu_2(G)=6$ holds if and only if
\[
 |G|=16\qquad\text{or}\qquad G\cong C_{18},C_{20},C_{22}.
\]
\item One has $\mu_2(G)\ge6$ for $|G|\ge16$ and
$\mu_2(G)\ge7$ for $|G|\ge23$.  Both thresholds are sharp.
In the cyclic case,
\[
 \kappa_2(17)=\kappa_2(19)=\kappa_2(21)
 =\kappa_2(24)=\kappa_2(26)=\kappa_2(28)=\kappa_2(30)=7.
\]
At order thirty the defining class contains a matrix with
exactly seven positive and twenty-three negative eigenvalues.
\end{enumerate}
\end{theorem}

The remaining cyclic values for $5\le n\le15$ are given in
Theorem~\ref{a:thm:small-table}.  More generally, a feasible positive
support $\{0,\pm r_1,\ldots,\pm r_d\}$ of $2d+1$ distinct elements
with $d\ge1$ satisfies a nonzero integer relation $\sum_i\alpha_i r_i=0$ with
$\sum_i|\alpha_i|\le\max\{5,d+2\}$.  For two or three nonzero
inverse pairs, the bound five is sharp.  Constant and alternating
vectors on a binary cube give the necessary mass estimate.
In the five-point case, the relation reduces both the cyclic and the
noncyclic supports to a finite list.  For noncyclic
$2$-groups, an index-two character gives a direct exclusion by
isolating a single positive coefficient in one block.  A six-point
support contains a positive involution; summing coefficients
over its cosets reduces the positive count to at most five.
Index-two convolution blocks then exclude most remaining
noncyclic groups.  An explicit integer spectrum settles the
last construction, on $C_2\times C_8$.

The cyclic bounds excluding five positive coefficients, together with the
values at $16,17,19,21$, appeared in the earlier preprint version
\href{https://arxiv.org/abs/2609.22744v1}{arXiv:2609.22744v1}.
They are proved here as part of the argument for the classification
of the value six.

\section{Fourier formulation and common bounds}

Let $G$ be a finite abelian group, written additively, and let
$\widehat G$ be its character group.  We use $G$ as the frequency
group and define
\[
 h(\chi)=\mathcal F^{-1}a(\chi)
 =\frac1{|G|}\sum_{k\in G}a_k\chi(k),\qquad
 (a*b)_k=\sum_{j\in G}a_jb_{k-j}.
\]
The matrix $(h(\chi\psi^{-1}))_{\chi,\psi\in\widehat G}$ has
eigenvalues $a_k$, and $\widehat{h^2}=|G|^{-1}(a*a)$.
Since a finite abelian group is isomorphic to its character
group, this gives the same minimum $\mu_2(G)$ as in the Introduction.
For a real even $a$, its inverse transform is real and even, and
\begin{equation}
 H_h^{\circ2}\succ0\quad\Longleftrightarrow\quad a*a>0.
 \label{a:eq:fourier-equivalence}
\end{equation}
Thus
\begin{equation}
 \mu_2(G)=\min\bigl\{|\{k:a_k>0\}|:
 a_{-k}=a_k\in\mathbb R,\ \mathcal F^{-1}a>0,\ a*a>0\bigr\}.
 \label{a:eq:kappa-def}
\end{equation}
We call such a spectrum $a$ feasible and write $P=\{k:a_k>0\}$
for its positive support.  Positive rescaling and automorphisms
of $G$ preserve feasibility.  We denote the trivial character
by $0$.  In the cyclic case $G=C_n=\mathbb Z/n\mathbb Z$, put
$\omega=e^{2\pi i/n}$ and identify the character indexed by $x$
with $k\mapsto\omega^{kx}$; then
$a_k=\sum_xh(x)\omega^{-kx}$ and
$h(x)=n^{-1}\sum_ka_k\omega^{kx}$.
\label{sec:fourier}

\begin{lemma}\label{a:lem:basic}
Let $a$ be feasible in \eqref{a:eq:kappa-def}, let
$P=\{k:a_k>0\}$, and put $h=\mathcal F^{-1}a$.  Then
\begin{enumerate}
\item $0\in P$ and $|a_k|<a_0$ for every $k\ne0$;
\item $h(0)>h(x)$ for every $x\ne0$;
\item $P$ generates $G$.
\end{enumerate}
\end{lemma}

\begin{proof}
One has $a_0=\sum_\chi h(\chi)>0$.  For $k\ne0$, the
character $\chi\mapsto\chi(k)$ is nonconstant and $h>0$, so
strictness in the triangle inequality gives $|a_k|<a_0$.
The kernel $h^2$ is strictly positive definite.  Its principal
submatrix indexed by $0,\chi$ gives
$h(0)^4-h(\chi)^4>0$ for $\chi\ne0$, hence $h(0)>h(\chi)$.
If $P$ did not generate $G$, a nontrivial character $\chi$
would be one on $P$.  Fourier inversion would then give
\[
 |G|(h(0)-h(\chi))
 =\sum_ka_k(1-\operatorname{Re}\chi(k))\le0,
\]
a contradiction.
\end{proof}

\begin{lemma}\label{lem:deficit}
Let $a$ be feasible on a nontrivial finite abelian group.  Put
$d_\chi=h(0)-h(\chi)$ and $\Delta=\sum_{k\ne0}a_k$.  Then
\[
 \Delta=\sum_\chi d_\chi>0,\qquad |a_k|\le\Delta\quad(k\ne0).
\]
\end{lemma}

\begin{proof}
The first identity follows by summing Fourier inversion.
For $k\ne0$, character orthogonality gives
$a_k=-\sum_\chi d_\chi\overline{\chi(k)}$, so the triangle
inequality gives the bound.
\end{proof}

Feasibility also makes the matrix $(a_{i-j})_{i,j\in G}$
positive definite: its eigenvalues are $|G|h(\chi)>0$.
If $0,r,s,r+s$ are distinct, its quadratic form on these coordinates
with vector $(1,-1,-1,1)$ gives
\begin{equation}
 4a_0-4a_r-4a_s+2a_{r-s}+2a_{r+s}>0.
 \label{eq:four}
\end{equation}

We will repeatedly pass to quotients.

\begin{lemma}\label{a:lem:quotient}
Let $\Gamma$ be a finite abelian group, let $K\le\Gamma$, and let
$\pi:\Gamma\to\Gamma/K$.  If $a$ is feasible on $\Gamma$, put
$P=\{y:a_y>0\}$ and define
\[
 b_{\bar x}=\sum_{\pi(y)=\bar x}a_y.
\]
Then $b$ is feasible on $\Gamma/K$ and
\begin{equation}
 \{\bar x:b_{\bar x}>0\}\subseteq\pi(P).
 \label{a:eq:quotient-support}
\end{equation}
\end{lemma}

\begin{proof}
The function $b$ is real and inversion-invariant.  For every character
$\psi$ of $\Gamma/K$, the composite $\psi\circ\pi$ is a character of
$\Gamma$, and
\[
 \mathcal F^{-1}b(\psi)
 =\frac1{|\Gamma/K|}\sum_{y\in\Gamma}a_y\psi(\pi(y))
 =|K|\mathcal F^{-1}a(\psi\circ\pi)>0.
\]
Moreover,
\[
 (b*b)_{\bar x}=\sum_{\pi(y)=\bar x}(a*a)_y>0.
\]
If a fibre does not meet $P$, all its $a$-coefficients are nonpositive,
which proves \eqref{a:eq:quotient-support}.
\end{proof}

\begin{lemma}\label{b:lem:single-positive-factor}
Let $G$ be a finite abelian group and let $u,v:G\to\mathbb R$ satisfy
\[
 u*v>0,\qquad \sum_Gv>0.
\]
If $v$ has exactly one positive coefficient, then $u>0$.
\end{lemma}

\begin{proof}
Translating $v$ only translates $u*v$, so we may assume that
$v(0)=B>0$ and $v(y)=-q(y)\le0$ for $y\ne0$.  Put
$Q=\sum_{y\ne0}q(y)$.  The positive total mass gives $Q<B$.
Choose $x_0$ at which $u$ attains its minimum.  If $u(x_0)\le0$, then
\[
 (u*v)(x_0)=Bu(x_0)-\sum_{y\ne0}q(y)u(x_0-y)
 \le(B-Q)u(x_0)\le0,
\]
contrary to $u*v>0$.  Thus $u>0$.
\end{proof}

\begin{lemma}[An index-two block with one positive coefficient]
\label{lem:parity-single}
Let $K$ be an index-two subgroup of a finite abelian group $G$,
choose $t\notin K$, and let $a$ be feasible on $G$.  On $K$ put
$e_x=a_x$ and $o_x=a_{t+x}$.  Then $e*o>0$ and both
$\sum_Ke_x$ and $\sum_Ko_x$ are positive.  If one of these
blocks has exactly one positive coefficient, every coefficient
in the other block is positive.  In this case $a$ has at least
$|G|/2+1$ positive coefficients.
\end{lemma}

\begin{proof}
For $x\in K$ one has $(a*a)_{t+x}=2(e*o)_x>0$.
Put $E=\sum e_x$ and $O=\sum o_x$.  The inverse transform
at the trivial character and at the nontrivial character of
$G/K$ gives $E+O>0$ and $E-O>0$, so $E>0$.
Also $EO=\sum e*o>0$, hence $O>0$.
Lemma~\ref{b:lem:single-positive-factor} gives the conclusion.
For $G=C_{2m}$ and $K=2G$, these are the usual parity blocks
$e_j=a_{2j}$ and $o_j=a_{2j+1}$ on $C_m$.
\end{proof}

Write $a=p-q$, where $p_k=\max(a_k,0)$ and $q_k=\max(-a_k,0)$,
and put $M=\sum_kq_k$.  We also write $p=a_+$ and $q=a_-$.
For a subset $R\subseteq G$, the identity
\[
 \sum_{k\in R}(p*q)_k
 =\sum_{j\in G}q_j\sum_{\substack{i\in P\\j+i\in R}}p_i
\]
allows us to bound the convolution sum by choosing, for each negative
position $j$, positive modes that translate it into $R$.

We write $\ip{f}{g}=\sum_{k\in G}f_kg_k$ and denote the indicator of
$R$ by $\one_R$.  The weight of a negative position $k$ in the
preceding identity is
\[
 W(k)=\sum_{j\in G}p_j\one_R(j+k),\qquad
 \sum_{t\in R}(p*q)_t=\sum_{k\in G}W(k)q_k.
\]

\begin{lemma}\label{lem:complement-pair}
Write $a=p-q$, where $p,q\ge0$ have disjoint supports, put
$P=\supp p$ and $M=\sum q_k$, and let
$\varnothing\ne R\subseteq G\setminus(P+P)$.
If $\sum_R(a*a)>0$, then $M>0$ and there are $k,l\in\supp q$ such that
\[
 k+l\in R,\qquad W(k)+W(l)<M.
\]
If $q$ is even, there also exist $k,l\in\supp q$ with
\begin{equation}
 W(k)+W(l)<\frac M2
 \bigl(\one_R(k+l)+\one_R(k-l)\bigr).
 \label{eq:complement-even-pair}
\end{equation}
In particular, strict positivity is impossible if there are a set
$E\subseteq G$ and numbers $0\le v\le u$ such that
\[
 (E+E)\cap R=\varnothing,\qquad M\le u+v,
 \qquad W(k)\ge
 \begin{cases}
 v,&k\in E,\\
 u,&k\notin E
 \end{cases}
 \quad(k\in\supp q).
\]
\end{lemma}

\begin{proof}
Since $p*p$ vanishes on $R$, the positive convolution sum gives
$\sum_Rq*q>2\sum_Rp*q$ and hence $M>0$.  Thus
\[
 \begin{aligned}
 \sum_{k,l}(W(k)+W(l))q_kq_l
 &=2M\sum_Rp*q\\
 &<M\sum_Rq*q
 =M\sum_{k,l}\one_R(k+l)q_kq_l.
 \end{aligned}
\]
A pair with $q_kq_l>0$ must satisfy
$W(k)+W(l)<M\one_R(k+l)$, which proves the first assertion.
When $q$ is even,
\[
 \sum_Rq*q=\frac12\sum_{k,l}
 \bigl(\one_R(k+l)+\one_R(k-l)\bigr)q_kq_l.
\]
Substituting this identity gives \eqref{eq:complement-even-pair}.
For the pair in the first assertion, the final hypotheses force
at least one member to lie outside $E$.  Its weight sum is therefore
at least $u+v\ge M$, a contradiction.
\end{proof}

\begin{lemma}\label{a:lem:three-positive}
For $n\ge6$, one has $\kappa_2(n)\ge4$.  If $n\ge7$ is odd, then
$\kappa_2(n)\ge5$.
\end{lemma}

\begin{proof}
If at most three coefficients are positive, Lemma~\ref{a:lem:basic} and
evenness reduce the positive support, after an automorphism, to
$P=\{0,\pm1\}$.  Write
\[
 p_0=A,\qquad p_{\pm1}=C.
\]
For $d_x=h(0)-h(x)$, summing over $G$ gives
\[
 0<\sum_xd_x=2C-M,
\]
so $M<2C$.  Let
\[
 R=G\setminus\{0,\pm1,\pm2\}.
\]
This set is nonempty for $n\ge6$.  Negative mass already in $R$
has zero-mode weight $A>C$, while each of $\pm2$ is translated
into $R$ by a mode of weight $C$.  Thus $W(k)\ge C$ throughout
$\supp q$, and $M<2C$ contradicts Lemma~\ref{lem:complement-pair}.
Thus $\kappa_2(n)\ge4$.  In odd order every nonzero inversion
orbit has two elements, so the number of positive coefficients is odd.
The second assertion follows.
\end{proof}

\begin{theorem}\label{a:thm:no-four-even}
If $n\ge10$ is even, then $\kappa_2(n)\ge5$.
\end{theorem}

\begin{proof}
By Lemma~\ref{a:lem:three-positive}, it suffices to exclude exactly four
positive coefficients.  Write $n=2m$.  Inversion symmetry forces their
support to be $P=\{0,m,\pm r\}$.  If $m\ge6$, reduction modulo $m$
and Lemma~\ref{a:lem:quotient} give a feasible spectrum on $C_m$ with
at most three positive coefficients, contrary to
Lemma~\ref{a:lem:three-positive}.  For $m=5$, the parity blocks have
one and three positive coefficients.  Lemma~\ref{lem:parity-single}
would give at least six, a contradiction.
\end{proof}

\begin{lemma}\label{a:lem:tail}
Let $r\ge0$ be an even function on $G$, and put $W=\sum_jr_j$.
If $k\ne0$ and $2k\ne0$, then
\[
 (r*r)_k\le\frac12\left(W^2-\sum_jr_j^2\right)\le\frac{W^2}{2}.
\]
\end{lemma}

\begin{proof}
Evenness gives
\[
 2(r*r)_k=\sum_jr_j(r_{j-k}+r_{j+k}).
\]
For each $j$, the two indices $j-k,j+k$ are distinct and both differ
from $j$, so the last sum is at most $\sum_jr_j(W-r_j)$.
\end{proof}

\begin{lemma}\label{a:lem:additive-basis}
Let $P\subseteq G$ contain $0$, be closed under inversion, and satisfy
$P+P=G$.  Then $\mu_2(G)\le|P|$.
\end{lemma}

\begin{proof}
Set $a_0=1$, set $a_k=\varepsilon>0$ on $P\setminus\{0\}$, and set
$a_k=0$ elsewhere.  Choose $\varepsilon$ so that
$\varepsilon(|P|-1)<1$.  Then
\[
 |G|\mathcal F^{-1}a(\chi)\ge1-\varepsilon(|P|-1)>0.
\]
Since $P+P=G$, every coefficient of $a*a$ is positive, and the
positive support is exactly $P$.  The zero coefficients may also be changed
simultaneously to sufficiently small negative numbers.  All strict
inequalities persist and the positive support is unchanged.
\end{proof}

\begin{lemma}\label{a:lem:strong-mass}
Suppose that $P=\{0,\pm r,\pm s\}$ is a five-point feasible positive
support on a finite abelian group.  Write $a=p-q$, with
$p_0=A$, $p_{\pm r}=C$, $p_{\pm s}=D$, and put
\[
 S=C+D,\qquad M=\sum q_k,\qquad
 E=\frac{a_{r-s}+a_{r+s}}2.
\]
Then
\begin{equation}
 M\le S+\min(C,D),
 \label{a:eq:strong-mass}
\end{equation}
and
\begin{equation}
 M\le S-E<A.
 \label{eq:mass-sandwich}
\end{equation}
In particular, $A>M$; if $r-s,r+s\notin P$, then also $A>S$.
\end{lemma}

\begin{proof}
Lemma~\ref{lem:deficit} gives $\Delta=2S-M\ge\max(C,D)$,
which proves \eqref{a:eq:strong-mass}.
The matrix $T=(a_{i-j})$ is positive definite, and
Lemma~\ref{a:lem:basic} shows that its largest eigenvalue is
$\rho=\sum a_k=A+2S-M$.
On the four distinct coordinates $0,r,s,r+s$, the vectors
$(1,1,1,1)$ and $(1,-1,-1,1)$ therefore give
\[
 A+S+E\le\rho,\qquad A-S+E>0,
\]
respectively.  These are exactly the two inequalities in
\eqref{eq:mass-sandwich}.
\end{proof}

\begin{lemma}\label{lem:boundary-reduction}
Let $a=p-q$ be feasible, put $P=\supp p$, $A=p_0$, and
$M=\sum q_k$, and let $\varnothing\ne R\subseteq G\setminus(P+P)$.
Write $q=q_0+r$, where $r=q\one_R$, and put
\[
 W=\sum r_k,\qquad K=\sum_R(p-A\one_{\{0\}})*r.
\]
Then
\begin{equation}
 \sum_R(q_0*q_0-2p*q_0)>2K+W^2+2(A-M)W.
 \label{eq:boundary-reduction}
\end{equation}
More generally, if $0\le\omega\le1$ and $\omega=1$ on $R$, then
\begin{equation}
 \ip{\omega}{(p-q_0)*(p-q_0)}>W^2+2(A-M)W.
 \label{eq:weighted-boundary}
\end{equation}
If $A\ge M$, both left-hand sides are strictly positive.
\end{lemma}

\begin{proof}
On $R$, the convolution $p*p$ vanishes and
$\sum_Rp*r=AW+K$.  Since $\sum q_0=M-W$,
\[
 \sum_Rq*q\le\sum_Rq_0*q_0+2(M-W)W+W^2.
\]
Substitution into $\sum_Ra*a>0$ gives
\eqref{eq:boundary-reduction}.
For the weighted assertion, use
\[
 \ip{\omega}{2q_0*r+r*r}\le2(M-W)W+W^2,
 \qquad \ip{\omega}{p*r}\ge AW.
\]
Expanding $a*a=(p-q_0)*(p-q_0)+2q_0*r+r*r-2p*r$
and using $\ip{\omega}{a*a}>0$ gives \eqref{eq:weighted-boundary}.
\end{proof}

\begin{lemma}\label{lem:cube-boundary}
Let $a=p-q$ be feasible on a finite abelian group $G$, with positive
support $P=\{0,t,\pm r,\pm s\}$ of six distinct elements and $2t=0$.
Suppose that the eight points
\[
 \varepsilon_1r+\varepsilon_2s+\varepsilon_3t,
 \qquad \varepsilon\in\{0,1\}^3,
\]
are distinct.  Let $R\subseteq G\setminus(P+P)$ contain the four
distinct elements $t\pm r\pm s$, and put
$q_0=q\one_{G\setminus R}$.  Then
\[
 \sum_R(q_0*q_0-2p*q_0)>0.
\]
\end{lemma}

\begin{proof}
Write $p_0=A$, $p_t=B$, $p_{\pm r}=C$, $p_{\pm s}=D$ and
$M=\sum q_k$.  By Lemma~\ref{a:lem:basic}, the positive definite
matrix $(a_{i-j})$ has largest eigenvalue
$\rho=A+B+2C+2D-M$.
On the eight displayed coordinates, let $Q_+,Q_-$ be the Rayleigh
quotients of the constant vector and the vector with signs
$(-1)^{\varepsilon_1+\varepsilon_2+\varepsilon_3}$.
Even-length differences cancel in their subtraction.  Each ordered
pair at odd Hamming distance has coefficient $2/8=1/4$ in $Q_+-Q_-$.
At distance one, each of $\pm r,\pm s$ occurs four times, whereas
$t=-t$ occurs eight times.  At distance three, each of $t\pm r\pm s$
occurs twice; evenness of $a$ therefore gives
\[
 Q_+-Q_-=2B+2C+2D+a_{t+r+s}+a_{t+r-s}.
\]
Since $Q_+\le\rho$ and $Q_->0$, putting $W=\sum_Rq$ gives
\[
 A-M>B-q_{t+r+s}-q_{t+r-s}\ge B-W/2>-W/2.
\]
Lemma~\ref{lem:boundary-reduction} now yields
\[
 \sum_R(q_0*q_0-2p*q_0)>W^2+2(A-M)W\ge0,
\]
also when $W=0$.
\end{proof}

\begin{theorem}\label{a:thm:short-relation}
Suppose that the positive support of a feasible spectrum on a finite
abelian group $G$ is
\[
 P=\{0,\pm r_1,\ldots,\pm r_d\},\qquad d\ge1,
\]
with all $2d+1$ displayed elements distinct.  Then there is a nonzero
$\alpha\in\Z^d$ such that
\begin{equation}
 \sum_{i=1}^d\alpha_i r_i=0,\qquad
 \sum_{i=1}^d|\alpha_i|\le\max\{5,d+2\}.
 \label{a:eq:short-relation}
\end{equation}
For $d=2$ and $d=3$, the bound five cannot be decreased.
\end{theorem}

\begin{proof}
Write $a=p-q$, with $p_0=A$, $p_{\pm r_i}=c_i$, and put
$S=\sum_i c_i$, $M=\sum q$.  Lemmas~\ref{a:lem:basic} and
\ref{lem:deficit} give $M<2S$ and show that the positive definite
matrix $(a_{x-y})$ has largest eigenvalue $\rho=A+2S-M$.
Suppose that \eqref{a:eq:short-relation} fails.  Put
\[
 \phi(u)=\sum_i u_i r_i,\qquad
 B_j=\{u\in\Z^d:|u|_1\le j\},\qquad R=G\setminus(P+P).
\]
The map $\phi$ is injective on $B_2$, and
$\phi(B_3\setminus B_2)$ is disjoint from $P+P=\phi(B_2)$.
Thus $R\ne\varnothing$.  Each negative position in $P+P$ has a
unique representative $u$ with $|u|_1=2$.  For every $i$, choose
$\sigma_i\in\{-1,1\}$ so that $|u+\sigma_i e_i|_1=3$, where
$e_i$ is the $i$th standard basis vector.  The corresponding shift
by $\sigma_i r_i$ enters $R$, so the transport weight is at least $S$.

The $2^d$ points $\phi(\varepsilon)$, $\varepsilon\in\{0,1\}^d$,
are distinct.  Let $Q_+,Q_-$ be the Rayleigh quotients of the
constant vector and the vector with signs $(-1)^{|\varepsilon|_1}$
on these points.  Each cube difference $\delta\in\{-1,0,1\}^d$
occurs $2^{d-|\delta|_1}$ times.  Since $Q_+\le\rho$ and $Q_->0$,
\[
 A-M>2E,\qquad
 E=\sum_{\substack{\delta\in\{-1,0,1\}^d\\
                   |\delta|_1\ge3\ \mathrm{odd}}}
          2^{-|\delta|_1}a_{\phi(\delta)}.
\]
Every frequency in this sum lies in $R$.  For each fixed frequency,
the sum of the weights $2^{-|\delta|_1}$ over its representations
appearing in $E$ is at most $1/4$, as follows.

For each representation $\delta$ of a fixed frequency, take the
face of $\{0,1\}^d$ fixing coordinate $i$ to $1$ or $0$ according
as $\delta_i=1$ or $-1$.  This face has $2^{d-|\delta|_1}$ vertices.
If two such faces fix $t$ coordinates oppositely, these coordinates
contribute $2t$ to $|\delta-\delta'|_1$, while each remaining coordinate
contributes at most one.  Thus $|\delta-\delta'|_1\le2t+(d-t)=d+t$.
Their distinct representations give
a nonzero relation of length at least $d+3$, so $t\ge3$.
The faces have Hamming distance at least three, and their
radius-one neighborhoods are disjoint.  A face of codimension
$m\ge3$ has $(m+1)2^{d-m}\ge4\cdot2^{d-m}$ vertices in this
neighborhood.  Summing the disjoint sizes gives the coefficient bound.
Consequently, with $W=\sum_Rq$,
\[
 E\ge-W/4,\qquad A-M>-W/2.
\]

Put $q_0=q\one_{G\setminus R}$ and $N=M-W$.  The transport bound
and $0\le N\le M<2S$ give
\[
 \sum_R(q_0*q_0-2p*q_0)\le N(N-2S)\le0.
\]
But Lemma~\ref{lem:boundary-reduction} gives
\[
 \sum_R(q_0*q_0-2p*q_0)>W^2+2(A-M)W\ge0,
\]
a contradiction, also when $W=0$ or $N=0$.

For sharpness at $d=2$, take $C_{13}$ and $P=\{0,\pm1,\pm5\}$.
Here $P+P=C_{13}$, so Lemma~\ref{a:lem:additive-basis} gives
feasibility.  The thirteen elements of $B_2$ have distinct images,
so $B_2-B_2=B_4$ contains no nonzero relation, whereas
$3+2\cdot5\equiv0\pmod{13}$.
For $d=3$, the feasible support $\{0,\pm1,\pm5,\pm8\}$ in $C_{30}$
is supplied by Proposition~\ref{prop:rational-constructions}.
Its twenty-five words of length at most two have distinct images.
The coefficient vector $(-3,-1,1)$ gives the relation
$-3\cdot1-5+8=0$ of length five.
\end{proof}

\section{Constructions and the small orders}\label{sec:constructions}

We first treat the cyclic case $G=C_n$; the noncyclic case begins
in Section~\ref{sec:noncyclic-five}.

All cyclic constructions used below are collected here.  In each case
the spectrum is real and even, so only one representative of an inverse
pair needs to be specified.

\begin{proposition}\label{prop:additive-constructions}
For each row of the following table, the set $P$ is a two-fold
additive basis of $C_n$, and there is a feasible spectrum whose
positive support is $P$.
\[
\begin{array}{c|l@{\qquad}c|l}
n&P&n&P\\ \hline
5&\{0,\pm1\}&6&\{0,3,\pm1\}\\
7,9&\{0,\pm1,\pm2\}&8&\{0,4,\pm1\}\\
10&\{0,\pm1,\pm4\}&11&\{0,\pm1,\pm3\}\\
12&\{0,\pm2,\pm3\}&13&\{0,\pm1,\pm5\}\\
16&\{0,8,\pm1,\pm3\}&17&\{0,\pm1,\pm2,\pm6\}\\
19&\{0,\pm1,\pm2,\pm7\}&21&\{0,\pm1,\pm2,\pm8\}
\end{array}
\]
\end{proposition}

\begin{proof}
Adding the listed residues gives $P+P=C_n$ in every case.
Lemma~\ref{a:lem:additive-basis} applies.  More explicitly, take
$a_0=1$, $a_k=1/12$ on $P\setminus\{0\}$, and $a_k=0$ elsewhere.
Then $n\mathcal F^{-1}a\ge1-(|P|-1)/12\ge1/2$ and $a*a>0$.
\end{proof}

\begin{proposition}\label{prop:rational-constructions}
There are feasible spectra with five positive coefficients in orders
$14,15$, with six in orders $18,20,22$, and with seven in orders
$24,26,28,30$.  At order $30$, the spectrum may be chosen to have
twenty-three negative coefficients.
\end{proposition}

\begin{proof}
In each of the following spectra, set $a_0=1$ and set every
unspecified coefficient equal to zero:
\begin{align*}
n=14:\quad &a_{\pm1}=\tfrac13,\quad a_{\pm4}=\tfrac3{10},
 &&a_{\pm2}=-\tfrac1{24},\quad a_{\pm5}=-\tfrac3{40};\\
n=15:\quad &a_{\pm2}=a_{\pm7}=\tfrac7{24},
 &&a_{\pm6}=-\tfrac2{25};\\
n=18:\quad &a_9=\tfrac3{20},\quad a_{\pm1}=\tfrac14,
 \quad a_{\pm6}=\tfrac9{20},
 &&a_{\pm7}=-\tfrac1{10};\\
n=20:\quad &a_{10}=\tfrac29,\quad a_{\pm1}=\tfrac5{12},
 \quad a_{\pm4}=\tfrac15,
 &&a_{\pm2}=-\tfrac1{15},\quad a_{\pm5}=-\tfrac1{16};\\
n=22:\quad &a_{11}=\tfrac{19}{200},\quad a_{\pm1}=\tfrac{153}{500},
 \quad a_{\pm4}=\tfrac{2837}{10000},
 &&a_{\pm3}=-\tfrac{217}{2500},\quad a_{\pm6}=-\tfrac7{2000}.
\end{align*}
For the four remaining orders take
\[
\begin{aligned}
n=24:\quad &a_{\pm1}=a_{\pm5}=\tfrac3{10},\quad a_{\pm8}=\tfrac18,\\
 &a_{\pm2}=a_{\pm10}=-\tfrac1{240},\quad
 a_{\pm3}=a_{\pm9}=-\tfrac1{32},\quad a_{\pm6}=-\tfrac1{13};\\[2pt]
n=26:\quad &a_{\pm1}=\tfrac14,\quad a_{\pm3}=\tfrac5{24},
 \quad a_{\pm8}=\tfrac{13}{50},\\
 &a_{\pm6}=-\tfrac1{200},\quad a_{\pm7}=-\tfrac4{63};\\[2pt]
n=28:\quad &a_{\pm1}=\tfrac5{28},\quad a_{\pm4}=\tfrac6{25},
 \quad a_{\pm11}=\tfrac9{26},\\
 &a_{\pm7}=-\tfrac2{25},\quad a_{\pm12}=-\tfrac2{43},
 \quad a_{14}=-\tfrac1{360};\\[2pt]
n=30:\quad &a_{\pm1}=\tfrac{29}{100},\quad
 a_{\pm5}=a_{\pm8}=\tfrac14,\\
 &a_{\pm2}=-\tfrac1{25},\quad a_{\pm6}=-\tfrac1{150},\quad
 a_{\pm9}=-\tfrac7{100},\quad a_{\pm13}=-\tfrac1{400}.
\end{aligned}
\]
We verify both strict positivity conditions by a common rational
calculation.  Put
\[
 T_j=\sum_{k=0}^{n-1}a_k\cos(2\pi jk/n)
     =n(\mathcal F^{-1}a)(j).
\]
For a rational $0<q<1/2$, define
\[
 L(q)=\sum_{r=0}^{9}\frac{(-1)^r(q\,355/113)^{2r}}{(2r)!},
 \qquad
 U(q)=\sum_{r=0}^{10}\frac{(-1)^r(q\,333/106)^{2r}}{(2r)!}.
\]
The inequalities $333/106<\pi<355/113$, monotonicity of cosine on
$[0,\pi]$, and the alternating Taylor bounds imply
\[
 L(q)\le\cos(q\pi)\le U(q).
\]
Reduce every angle modulo $2\pi$ and use symmetry to express its
cosine as $\sigma\cos(q\pi)$ with $\sigma\in\{-1,1\}$ and
$0\le q\le1/2$.  At $q=0,1/2$ use the exact values.
For the other terms, use $L(q)$ when the resulting coefficient is
positive and $U(q)$ when it is negative.  This gives a rational lower
bound $\ell_j$ for $T_j$.

Direct substitution in these finite sums and in
$(a*a)_t=\sum_k a_ka_{t-k}$ gives the following table.  The third
column is the exact minimum over all convolution positions; the
fourth gives a rational number $\beta$ such that
$T_j\ge\ell_j>\beta$ for every $j$.
\begin{equation}
\renewcommand{\arraystretch}{1.5}
\begin{array}{c|c|c|c}
n&|P|&\min_t(a*a)_t&\beta\\ \hline
14&5&\frac{1}{360}&\frac{1}{35}\\
15&5&\frac{4}{625}&\frac{1}{150}\\
18&6&\frac{1}{100}&\frac{1}{20}\\
20&6&\frac{1}{120}&\frac{1}{1100}\\
22&6&\frac{3339}{6250000}&\frac{1}{2000}\\
24&7&\frac{153901}{9734400}&\frac{1}{225}\\
26&7&\frac{13}{25200}&\frac{1}{600}\\
28&7&\frac{193}{29250}&\frac{1}{125}\\
30&7&\frac{23}{30000}&\frac{1}{200}\\

\end{array}
\label{eq:construction-certificates}
\end{equation}
Every entry in the last two columns is positive.  Thus all nine
spectra are feasible and have the asserted numbers of positive
coefficients.  At order $30$, changing all zero coefficients, in
inverse pairs and at the half period, to sufficiently small negative
values preserves the two strict positivity conditions.  The resulting
spectrum has seven positive and twenty-three negative coefficients.
\end{proof}

\begin{theorem}\label{a:thm:small-table}
For $5\le n\le15$,
\[
\begin{array}{c|ccccccccccc}
n&5&6&7&8&9&10&11&12&13&14&15\\ \hline
\kappa_2(n)&3&4&5&4&5&5&5&5&5&5&5.
\end{array}
\]
\end{theorem}

\begin{proof}
The upper bounds follow from Propositions~\ref{prop:additive-constructions}
and \ref{prop:rational-constructions}.  The lower bounds for odd
$n\ge7$ and for $n=6,8$ follow from
Lemma~\ref{a:lem:three-positive}; those for even $n\ge10$ follow
from Theorem~\ref{a:thm:no-four-even}.  At $n=5$, the generating
positive support contains zero and at least one inverse pair, so it
has at least three elements.
\end{proof}

\section{Five positive Fourier coefficients}

Suppose in this section that $a$ is feasible on $C_n$ and has exactly
five positive coefficients.  Lemma~\ref{a:lem:basic} gives
\begin{equation}
 P=\{0,\pm r,\pm s\},\qquad (r,s,n)=1.
 \label{a:eq:five-support}
\end{equation}
The five displayed elements are distinct.  For integer frequencies
$i,j$, we use the notation $\cP_{ij}=\{0,\pm i,\pm j\}$.
Write
\begin{equation}
 p_0=A,\qquad p_{\pm r}=C,\qquad p_{\pm s}=D,
 \qquad a=p-q,\qquad M=\sum_kq_k.
 \label{a:eq:five-masses}
\end{equation}

\begin{lemma}\label{a:lem:parity}
If $n\ge10$ is even, the two frequencies $r,s$ in
\eqref{a:eq:five-support} have opposite parity.
\end{lemma}

\begin{proof}
The frequencies $r,s$ cannot both be even, since $P$ generates $C_n$.
If they were both odd, the even block would have only the positive
zero coefficient.  Lemma~\ref{lem:parity-single} would then give at
least $n/2+1\ge6$ positive coefficients, a contradiction.
\end{proof}

\begin{lemma}\label{b:lem:consecutive-pairs}
The positive support $\{0,\pm1,\pm2\}$ is infeasible on $C_n$ for
every $n\ge10$.  The threshold ten is sharp.
\end{lemma}

\begin{proof}
Write $a=p-q$, with
\[
 p_0=A,\qquad p_{\pm1}=C,\qquad p_{\pm2}=D,
 \qquad q_{\pm3}=u,\quad q_{\pm4}=v,
\]
and put $S=C+D$, $s=u+v$, and $M=\sum q_k$.
Lemma~\ref{a:lem:strong-mass} gives $A>M$ and
$M\le S+\min(C,D)$, while $A>\max(C,D)$.
For $n=10$ one has
\[
 (a*a)_5=-2Aq_5-4Cv-4Du\le0,
\]
so assume $n\ge11$.  Let $R=\{5,\ldots,n-5\}$ and let $q_0$
be the restriction of $q$ to $\{\pm3,\pm4\}$.
The boundary pairs $\pm3,\pm4$ enter $R$ with weights $D,S$,
respectively, whereas opposite-sign boundary pairs never sum into $R$.
Consequently
\[
 \sum_Rp*q_0=2Du+2Sv,\qquad
 \sum_Rq_0*q_0\le2s^2.
\]
Lemma~\ref{lem:boundary-reduction} yields
\begin{equation}
 s^2>2Du+2Sv.
 \label{b:eq:consecutive-boundary}
\end{equation}
In particular $s>2D$, and $2s\le M\le C+2D$ then gives $C>2D$.
Since $q\ne0$, positive definiteness of
$(a_{i-j})$ gives
\[
 0<\ip{q}{a*q}=\ip{q}{p*q}-\ip{q}{q*q}.
\]
Combining this with $\ip{q}{a*a}>0$ and evenness, we obtain
\[
 0<\ip{q}{p*p}-\ip{q}{p*q}
   =\ip{q}{p*p}-\ip{p}{q*q}.
\]
But
\[
 \ip{q}{p*p}=4CDu+2D^2v,\qquad
 \ip{p}{q*q}\ge2A(u^2+v^2)+4Cuv>2Cs^2.
\]
Thus $s^2<2Du+(D^2/C)v\le2Du+2Sv$, contradicting
\eqref{b:eq:consecutive-boundary}.
Finally, $P+P=C_9$, so Lemma~\ref{a:lem:additive-basis}
shows that the threshold is sharp.
\end{proof}

The short relation gives a common support reduction for all orders at least sixteen.

\begin{proposition}\label{a:prop:six-families}
Let $n\ge16$.  If a feasible spectrum on $C_n$ has five positive
coefficients, then, up to an automorphism, its positive support is
one of
\begin{equation}
 \cP_{12},\qquad\cP_{13},\qquad\cP_{14},\qquad\cP_{23},
 \label{a:eq:six-families}
\end{equation}
where $\cP_{13}$ occurs only when $n$ is odd, or one of the
additional supports
\begin{equation}
 \begin{aligned}
 n=20:&\qquad \cP_{18},\ \cP_{25},\ \cP_{45},\\
 n=25:&\qquad \cP_{15},\ \cP_{1,10}.
 \end{aligned}
 \label{a:eq:exceptions}
\end{equation}
\end{proposition}

\begin{proof}
Write $P=\{0,\pm r,\pm s\}$.  The support generates $C_n$.
Quotienting by the subgroup generated by either $t\in\{r,s\}$
leaves at most three positive positions in a cyclic group of
order $\gcd(t,n)$.  Lemmas~\ref{a:lem:quotient} and
\ref{a:lem:three-positive} therefore give
\[
 \gcd(r,n)\le5,\qquad\gcd(s,n)\le5.
\]
Theorem~\ref{a:thm:short-relation} supplies a nonzero relation
$\alpha r+\beta s=0$ with $|\alpha|+|\beta|\le5$.

Suppose first that $\alpha\beta\ne0$.  If the coefficients
are not coprime, their absolute values must be $(2,2)$.
For odd $n$, this gives $r=\pm s$, contrary to the five
distinct support elements.  For $n=2m$, it gives
$r\pm s=0$ or $m$.  The first possibility is excluded by
distinctness; in the second, reduction modulo $m$ identifies
the inverse pairs.  The quotient has at most three positive
positions and $m\ge8$, contrary to
Lemma~\ref{a:lem:three-positive}.

Thus $\alpha,\beta$ are coprime.  Excluding $(1,1)$ by
distinctness, their possible absolute values, up to interchange,
are $(1,2),(1,3),(1,4),(2,3)$.
Choose $x,y$ with $xr+ys=1\pmod n$, and put
$u=\alpha y-\beta x$.  Then
\[
 ur=-\beta,\qquad us=\alpha\pmod n.
\]
Every common divisor of $u,n$ divides $\alpha$ and $\beta$,
so $u$ is a unit.  Multiplication by $u$ gives the four
supports in \eqref{a:eq:six-families}.  If $n$ is even,
Lemma~\ref{a:lem:parity} excludes $\cP_{13}$.

If one relation coefficient is zero, the corresponding frequency
$t$ has order $q\in\{3,4,5\}$.  Since $n/q=\gcd(t,n)\le5$
and $n\ge16$, the only possibilities are
\[
 (n,q)=(16,4),\ (20,4),\ (20,5),\ (25,5).
\]
At $(16,4)$, the order-four pair is $\{\pm4\}$ and the
other frequency is odd.  Normalizing it to one preserves this
pair and gives $\cP_{14}$.

At $(20,4)$, the order-four pair is $\{\pm5\}$.
Lemma~\ref{a:lem:parity} makes the other frequency even,
and generation makes it coprime to five.  Its gcd with twenty
is therefore two or four.  The units $1,3,7,9$ act transitively
on each of these two gcd classes and preserve $\{\pm5\}$,
giving $\cP_{25}$ or $\cP_{45}$.

At $(20,5)$, the order-five pair is $\{\pm4\}$ or
$\{\pm8\}$ and the other frequency is odd.  If it is a
unit, normalization gives $\cP_{14}$ or $\cP_{18}$.
Otherwise it is $\pm5$, giving $\cP_{45}$ or $\cP_{58}$;
multiplication by three sends the latter to $\cP_{45}$.
Finally, at $(25,5)$, generation makes the other frequency
a unit.  Normalizing it to one leaves the order-five pair
as $\{\pm5\}$ or $\{\pm10\}$, giving the last two supports.
\end{proof}

\section{Exclusion of five-point supports}\label{sec:five-exclusions}
\subsection{\texorpdfstring{The support $\{0,\pm2,\pm3\}$}{The support with frequencies 2 and 3}}

Lemma~\ref{b:lem:consecutive-pairs} has already excluded $\cP_{12}$.
We next treat $\cP_{23}$.

\begin{theorem}\label{a:thm:P23}
The support $\cP_{23}$ is infeasible in every order $n\ge16$.
\end{theorem}

\begin{proof}
Write $p_{\pm3}=C$, $p_{\pm2}=D$, and divide by $C+D$.  Put
\[
 p_0=\alpha,\qquad p_{\pm3}=c,\qquad p_{\pm2}=d,
 \qquad c+d=1,\qquad M=\sum q_k.
\]
Let $q_0$ be the restriction of $q$ to
$\{\pm1,\pm4,\pm5,\pm6\}$, with
\[
 q_{\pm1}=u,\quad q_{\pm4}=v,\quad q_{\pm5}=w,\quad q_{\pm6}=x,
 \qquad t=v+w+x,\quad s=u+t.
\]
Put $R=C_n\setminus(P+P)=\{7,\ldots,n-7\}$,
$W=\sum_Rq$.  Lemma~\ref{a:lem:strong-mass} gives
$\alpha>\max(1,M)$ and
\[
 \frac34(u+w)+v+x+W/2\le\frac12.
\]
Consequently $s\le2/3$ and $v\le1/2$.

Lift $q_0$ to the even function on the integers supported on
$\pm1,\pm4,\pm5,\pm6$.  Its self-convolution has mass
$2(t^2+2ux)$ outside $[-6,6]$.
Reduction modulo $n$ can only decrease this mass.  Since the boundary
pairs $\pm4,\pm5,\pm6$ enter $R$ with weights $c,1,1$,
Lemma~\ref{lem:boundary-reduction} gives
\[
 Q:=t^2+2ux-2cv-2(w+x)>0.
\]
Writing $k=w+x$, we have
\[
 Q-(v^2-2cv)
 =k(2v+k-2)+2ux\le k(2s-2)\le-2k/3.
\]
Thus $v>2c$, and hence $c<1/4$.

We now test the single weight
\[
 \omega=\one_R+\tfrac12\one_{\{\pm6\}}.
\]
For the integer convolution of the lifted function $q_0$, assign
weight one outside $[-6,6]$, weight $1/2$ at $\pm6$, and zero
elsewhere.  Its weighted mass is $2(t^2+2ux)+2uw$.
Every sum lies in $[-12,12]$; a sum with absolute value greater than
six has weight one before reduction and weight at most one afterwards.
Sums in $[-6,6]$ retain their weights.  Therefore
\[
 \ip{\omega}{q_0*q_0}\le2(t^2+2ux)+2uw.
\]
Since $\alpha>M$, the weighted assertion of
Lemma~\ref{lem:boundary-reduction} gives
$\ip{\omega}{(p-q_0)*(p-q_0)}>0$.

Finally, $\ip{\omega}{p*p}=c^2$ and
\[
 \ip{\omega}{p*q_0}\ge2cv+2(w+x)+\alpha x+dv.
\]
Consequently
\[
\begin{aligned}
 0<\ip{\omega}{(p-q_0)*(p-q_0)}
 &\le c^2-2dv-2\alpha x+2uw+2Q\\
 &\le c^2-2v(1+c-v)+(4s-4)k-2\alpha x\\
 &\le c^2-2v(1+c-v).
\end{aligned}
\]
Here the second bound uses
$4vk+2k^2+4ux+2uw\le4sk$.
The function $v(1+c-v)$ increases on $[2c,1/2]$.
Since $v>2c$ and $c<1/4$, the last expression is less than
$c^2-4cd<0$, a contradiction.
\end{proof}

\subsection{\texorpdfstring{The support $\{0,\pm1,\pm4\}$}{The support with frequencies 1 and 4}}

We use the complementary convolution sum and the weighted
boundary reduction.

\begin{theorem}\label{a:thm:P14-generic}
The support $\{0,\pm1,\pm4\}$ is infeasible in $C_n$ for
$n\ge16$, $n\ne17$.
\end{theorem}

\begin{proof}
Write $p_0=A$, $p_{\pm1}=C$, $p_{\pm4}=D$, and set
$S=C+D$, $M=\sum q$.  Lemma~\ref{a:lem:strong-mass} gives
$A>\max(S,M)$ and $M\le S+\min(C,D)$.
Let $R=C_n\setminus(P+P)$ and $E=\{\pm2,\pm3\}$.
The transport weights are at least $D$ on $E$ and $S$ elsewhere,
except at
\[
 Z=\begin{cases}\{8\},&n=16,\\\{\pm8\},&n=20,\\
                  \varnothing,&\text{otherwise},\end{cases}
\]
where they are at least $2C$.
These weights follow from $2+4=6$, $3+4=7$, $5+1=6$, $5+4=9$;
at $8$ use both shifts by $\pm1$ in the exceptional orders and
otherwise the shifts by $1,4$.  Mass in $R$ has weight at least $A>S$.
The folded pair of Lemma~\ref{lem:complement-pair} must belong to
$E\times E$.  Indeed, pairs in $Z\times Z$ have neither folded sum
in $R$.  Weight sums between the three classes are at least
$D+2C$, $D+S$, and $2C+S$, all at least $M$; pairs with both
members outside $E\cup Z$ have weight sum at least $2S\ge M$.
At most one folded sum of two members of $E$ is in $R$.
Hence $M>4D$.  Since $M\le S+\min(C,D)\le C+2D$, we obtain
$C>2D$, so every weight outside $E$ is now at least $S$.

Let $q_0$ be the restriction of $q$ to its boundary positions
$\{\pm2,\pm3,\pm5,\pm8\}$.  Write
\[
 q_0=u\one_{\{\pm2\}}+v\one_{\{\pm3\}}+b,\qquad \sum b=2t.
\]
Thus $b$ is even and supported on $\{\pm5,\pm8\}$, including
the single self-inverse position $8$ when $n=16$.
Since $q_5\le t$, the mass sandwich gives
\[
 A>S+v/2,\qquad 2u+\tfrac32(v+t)\le S.
\]
For two positions in $\supp b$, at most one folded sum is in $R$,
because the other is $0$ or $\pm3$.  The shifts by $\pm2$ meet
$R$ at most twice, and those by $\pm3$ at most once.
Consequently boundary reduction requires
\[
 0<v^2+t^2+4ut+2vt-2D(u+v)-2St.
\]
The mass bound gives $t+4u+2v\le2S-v-2t$, so the last expression
is at most $v^2-2D(u+v)-t(v+2t)$.
Thus $v^2>2D(u+v)$ and $v>2D$.

Now put $\omega=\one_R+\one_{\{\pm3\}}$.
Since $A>M$ and $\omega=1$ on $R$, the weighted form of
Lemma~\ref{lem:boundary-reduction} gives
$\ip{\omega}{(p-q_0)*(p-q_0)}>0$.
The shifts of $b$ by $\pm2$ have total
$\omega$-weight at most two, and those by $\pm3$ at most one.
Together with $0\le\omega\le1$, this gives
\[
 \ip{\omega}{q_0*q_0}\le2v^2+4t(t+2u+v).
\]
Furthermore,
\[
 \ip{\omega}{p*p}=4CD,\qquad
 \ip{\omega}{p*q_0}\ge2D(u+v)+2St+2Av+2Cu.
\]
Using the lower bound on $A$, the weighted convolution is less than
\[
 4CD-4Su-4(C+2D)v
 +4t(t+2u+v-S).
\]
The last term is nonpositive by the mass bound.
Since $v>2D$, this gives
\[
 0<\ip{\omega}{(p-q_0)*(p-q_0)}
 <4CD-4(C+2D)2D=-4CD-16D^2<0,
\]
a contradiction.
\end{proof}

\subsection{The order-twenty exceptions}

\begin{proposition}\label{a:prop:exceptions}
The supports $\cP_{18}$, $\cP_{25}$, and $\cP_{45}$ are infeasible in $C_{20}$.
\end{proposition}

\begin{proof}
For $\cP_{18}$, reduction modulo $10$ gives a feasible spectrum with
positive support contained in $\cP_{12}$.  If all five positions are
positive, Lemma~\ref{b:lem:consecutive-pairs} applies; otherwise
Theorem~\ref{a:thm:no-four-even} gives a contradiction.

It remains to exclude $\cP_{25}$ and $\cP_{45}$.
Give the even-frequency pair weight $C$ and the pair $\{\pm5\}$
weight $D$, and set $S=C+D$ and $\ell=\min(C,D)$.
Lemma~\ref{a:lem:strong-mass} gives $M\le S+\ell$ and $A>S$.
Let $R=G\setminus(P+P)$ and use the weights of
Lemma~\ref{lem:complement-pair}.
For $P=\{0,\pm2,\pm5\}$, the only possible weight below $S$
is $W(10)=2C\ge\ell$.  Take $E=\{10\}$; then $E+E=\{0\}$.
For $P=\{0,\pm4,\pm5\}$, the possible exceptions are
\[
 W(\pm8)=2D,\qquad W(10)=2C.
\]
If $C\ge D$, take $E=\{\pm8\}$, whose self-sum is
$\{0,\pm4\}\subseteq P$; if $D\ge C$, take $E=\{10\}$.
In every case $W\ge\ell$ on $E$, $W\ge S$ elsewhere, and
$(E+E)\cap R=\varnothing$.  Lemma~\ref{lem:complement-pair},
with $(u,v)=(S,\ell)$, excludes both supports.
\end{proof}

\begin{corollary}\label{a:cor:n18}
For every even $n\ge16$, one has $\kappa_2(n)\ge6$.
\end{corollary}

\begin{proof}
Theorem~\ref{a:thm:no-four-even} excludes four or fewer positive
coefficients.  Proposition~\ref{a:prop:six-families} leaves only
$\cP_{12}$, $\cP_{23}$, $\cP_{14}$, and the three order-twenty
exceptions.  The first three are excluded by
Lemma~\ref{b:lem:consecutive-pairs} and
Theorems~\ref{a:thm:P23} and \ref{a:thm:P14-generic}; Proposition~\ref{a:prop:exceptions}
excludes the remaining supports.
\end{proof}

\subsection{The remaining odd-order supports}

The preceding bounds for $\cP_{12}$, $\cP_{23}$ and $\cP_{14}$
also apply in odd order, with $n=17$ still to be treated for
$\cP_{14}$.  We now exclude $\cP_{13}$, that order-seventeen
case, and the two additional supports in order twenty-five.

In the remaining arguments put $a=p-q$, with $p=a_+$ and $q=a_-$.
For a five-point positive support write $p_0=A$, give the two nonzero
inverse pairs the weights $C,D$, and put
\[
 S=C+D,\qquad \ell=\min(C,D),\qquad M=\sum_kq_k.
\]
Thus $M\le S+\ell<2S$.  Since $|P+P|\le13<n$, the set
$R=C_n\setminus(P+P)$ is nonempty, and feasibility implies $M>0$.

\begin{lemma}\label{a:lem:odd13}
The support $\cP_{13}$ is infeasible in every odd order $n\ge17$.
The support $\cP_{1,10}$ is infeasible in $C_{25}$.
\end{lemma}

\begin{proof}
For $\cP_{13}$, put $p_{\pm1}=C$ and $p_{\pm3}=D$.
Lemma~\ref{a:lem:strong-mass} gives $A>S$, and
\[
 P+P=\{0,\pm1,\pm2,\pm3,\pm4,\pm6\}.
\]
For positive representatives of its negative boundary positions,
the following shifts enter $R$:
\[
 2+3=5,\qquad
 4+1=5,\quad4+3=7,\qquad
 6-1=5,\quad6+3=9.
\]
They remain outside $P+P$ at order $17$ and at every larger odd order.
Thus the positions $\pm2$ have weight at least $D\ge\ell$, the other
boundary positions have weight at least $S$, and mass already in $R$
has the zero-mode weight $A>S$.  Apply Lemma~\ref{lem:complement-pair} with $(u,v)=(S,\ell)$ and
$E=\{\pm2\}$, since $E+E=\{0,\pm4\}\subseteq P+P$.

For $\cP_{1,10}$ in $C_{25}$, let $p_{\pm1}=C,p_{\pm10}=D$.
Again $A>S$, and
\[
 P+P=\{0,\pm1,\pm2,\pm5,\pm9,\pm10,\pm11\}.
\]
The boundary positions $\pm2$ have weight at least $C+2D$.
The shifts
$9-1=8$, $9+10=19$, $11+1=12$, and $11+10=21$ give weight $S$
at $\pm9,\pm11$.  The only remaining boundary positions $\pm5$
have weight $2C$ from shifts by $\pm1$, and
$\{\pm5\}+\{\pm5\}=\{0,\pm10\}\subseteq P$.
Lemma~\ref{lem:complement-pair} applies with $E=\{\pm5\}$ and
$(u,v)=(S,\ell)$.
\end{proof}

\begin{lemma}\label{a:lem:o17}
The support $\cP_{14}$ is infeasible in $C_{17}$.
\end{lemma}

\begin{proof}
Take $p_{\pm1}=C,p_{\pm4}=D$.  Multiplication by the unit four
exchanges these inverse pairs, so assume $C\ge D$.  Divide by
$C+D$ and write
\[
 p_0=\alpha>1,\qquad p_{\pm1}=c,\quad p_{\pm4}=d,
 \qquad c+d=1,\quad c\ge d>0.
\]
For the negative part put
\[
 q_{\pm2}=u,\quad q_{\pm3}=v,\quad q_{\pm5}=w,\quad
 q_{\pm6}=x,\quad q_{\pm7}=y,\quad q_{\pm8}=z.
\]
Here Lemma~\ref{a:lem:strong-mass} gives $\alpha>1$ and, using
$E=-(v+w)/2$ in \eqref{eq:mass-sandwich},
\begin{equation}
 u+z+\frac34(v+w)+x+y\le\frac12.
 \label{a:eq:o17mass}
\end{equation}
In particular, $u+v+w+x+y+z\le2/3$.
Expanding $(a*a)_6+(a*a)_7$ and $(a*a)_3/2$, and replacing
$\alpha>1$ by $1$, gives the necessary inequalities
\begin{align}
 R={}&v^2+w^2+2uw+4uz+2vz+2wx+2vy\notag\\
    &-2d(u+v)-2c(w+z)-4(x+y)>0,\label{a:eq:o17a}\\
 F={}&cd-v-cu-dy+uw+vx+wz+xz+y^2/2>0.\label{a:eq:o17b}
\end{align}
The combination $R+vF$ cancels $v^2$.  To estimate its remaining
terms, put $T=u+z$, $V=v+w$, and
\[
 H=w^2+2uw+4uz+2vz+vwT.
\]
Since $uz\le T(du+cz)$ and $2T+3V/2\le1$, we have
\[
 4uz-2du-2cz\le-3V(du+cz).
\]
Also, \eqref{a:eq:o17mass} and $T\le1/2$ give
\[
 w^2+2uw+vwT
 \le w(1-v-w/2-2z)\le w\le2cw.
\]
Consequently, using $2-3c=3d-1\le d$ and $z\le1/2$,
\begin{equation}
 H-2du-2c(w+z)
 \le(2-3c)vz\le dv/2.
 \label{a:eq:o17quad}
\end{equation}
Finally,
\[
\begin{aligned}
 R+vF={}&H-2du-2c(w+z)-d(1+d)v-cuv\\
 &+x[-4+2w+v^2+vz]\\
 &+y[-4+(2-d)v+vy/2].
\end{aligned}
\]
Both bracketed coefficients are negative: their positive terms are
at most $2(w+v+z)\le4/3$ and $2(v+y)\le4/3$, respectively.
Thus \eqref{a:eq:o17quad} gives
$R+vF\le-d(1/2+d)v\le0$, contradicting $R>0$, $F>0$, and $v\ge0$.
\end{proof}

\begin{lemma}\label{a:lem:odd15}
The support $\cP_{15}$ is infeasible in every odd order $n\ge17$.
\end{lemma}

\begin{proof}
If this support were feasible, Theorem~\ref{a:thm:short-relation}
would give $\alpha+5\beta\equiv0\pmod n$ with
$0<|\alpha|+|\beta|\le5$.  The integer $\alpha+5\beta$ cannot be
zero, and its absolute value is at most $25$.  Listing the values
at least $17$ with odd absolute value leaves
$n\in\{17,19,21,25\}$.  In order $17$, multiplication by $3$ sends
$\cP_{15}$ to $\cP_{23}$, excluded by Theorem~\ref{a:thm:P23}.
In orders $19$ and $21$, multiplication by $4$ sends it to
$\cP_{14}$, excluded by Theorem~\ref{a:thm:P14-generic}.

For $n=25$, put $p_{\pm1}=C,p_{\pm5}=D$.  Then $A>S$ and
\[
 P+P=\{0,\pm1,\pm2,\pm4,\pm5,\pm6,\pm10\}.
\]
The positions $2,4,6$ enter the complement $R$ by the shifts
$(+1,+5)$, $(-1,+5)$, $(+1,+5)$, respectively, each with total
weight $S$.  The positions $\pm10$ enter $R$ by both shifts $\pm1$,
with weight $2C$.  Mass already in $R$ has the zero-mode weight
$A>S$.  Since $E=\{\pm10\}$ satisfies $E+E=\{0,\pm5\}\subseteq P$,
Lemma~\ref{lem:complement-pair} applies with $(u,v)=(S,\ell)$.
\end{proof}

\begin{corollary}\label{cor:base-bounds}
One has $\kappa_2(n)\ge6$ for every $n\ge16$, and
$\kappa_2(n)\ge7$ for every odd $n\ge17$.
\end{corollary}

\begin{proof}
For even $n\ge16$, the first assertion is Corollary~\ref{a:cor:n18}.

For odd $n\ge17$, Proposition~\ref{a:prop:six-families} lists all possible
five-point supports.  They are excluded by
Lemma~\ref{b:lem:consecutive-pairs}, Lemma~\ref{a:lem:odd13},
Theorems~\ref{a:thm:P14-generic} and \ref{a:thm:P23}, and
Lemmas~\ref{a:lem:o17} and \ref{a:lem:odd15}.
Lemma~\ref{a:lem:three-positive} supplies the bound five, and the
number of positive coefficients in odd order is odd.  Excluding
five therefore gives the bound seven.
\end{proof}

\section{Half-period descent and the finite reduction}\label{b:sec:descent}

The uniform lower bound can be strengthened by passing from an even
order to its half order.

\begin{theorem}\label{b:thm:half-period-descent}
Let $a$ be feasible on $\Z/2m\Z$ and define
\begin{equation}
 b_j=a_j+a_{j+m},\qquad j\in\Z/m\Z.
 \label{b:eq:half-period-descent}
\end{equation}
Then $b$ is feasible on $\Z/m\Z$, and
\[
 \{j:b_j>0\}\subseteq
 \pi\bigl(\{k:a_k>0\}\bigr),
\]
where $\pi:\Z/2m\Z\to\Z/m\Z$ is the quotient map.  If $a$ has at most
six positive coefficients, then $b$ has at most five.
\end{theorem}

\begin{proof}
The feasibility and support assertions follow from
Lemma~\ref{a:lem:quotient} for reduction modulo $m$.

The assertion is immediate when $a$ has at most five positive
coefficients.  Suppose that it has six.  The zero frequency is
positive.  All non-self-inverse positive frequencies occur in inverse
pairs, while the only self-inverse frequencies are $0$ and $m$.
Parity therefore forces $a_m>0$.  The quotient identifies $0$ and $m$,
so the image of the six positive positions has at most five elements.
\end{proof}

\begin{corollary}\label{b:cor:seven-even}
\[
 \kappa_2(2m)\le6\quad\Longrightarrow\quad\kappa_2(m)\le5.
\]
Consequently,
\[
 \kappa_2(n)\ge7\qquad(n\ge32,\ n\text{ even}).
\]
\end{corollary}

\begin{proof}
The first assertion is Theorem~\ref{b:thm:half-period-descent}.  For the
second, write $n=2m$.  Then $m\ge16$, and
Corollary~\ref{cor:base-bounds} gives $\kappa_2(m)\ge6$.
\end{proof}

When the half order is odd, the single-positive-factor lemma also
restricts the two parity blocks.

\begin{theorem}\label{b:thm:three-three}
Let $m\ge7$ be odd.  If a feasible real even spectrum on
$\Z/2m\Z$ has exactly six positive coefficients, then its even and odd
frequency blocks each have exactly three positive coefficients.
Equivalently, if its positive support is
\[
 \{0,m,\pm r,\pm s\},
\]
then $r$ and $s$ have opposite parity.
\end{theorem}

\begin{proof}
Since the total number of positive coefficients is even, both
self-inverse frequencies $0$ and $m$ belong to the positive support.
The oddness of $m$ puts one in each parity block, while all other
positive frequencies occur in inverse pairs within their respective
blocks.  Both block counts are therefore odd.  A block count of one
would give at least $m+1\ge8$ positive coefficients by
Lemma~\ref{lem:parity-single}.  Hence both counts must be three.
\end{proof}

Corollary~\ref{b:cor:seven-even} and the odd-order bound in
Corollary~\ref{cor:base-bounds} leave only $24,26,28,30$ to consider
for the lower bound seven in all orders at least $23$.

\section{The orders twenty-four and twenty-eight}\label{b:sec:24-28}

The half-order quotients in these two cases have even order.
Theorem~\ref{b:thm:half-period-descent} and
Theorem~\ref{a:thm:small-table} give five-point quotient supports.
Lemmas~\ref{a:lem:parity} and \ref{b:lem:consecutive-pairs}, together
with the small-order exclusions below, reduce the possible lifts to
a short list.  We first identify that list, also recording the
corresponding possibilities in order $20$, and then exclude its
supports in orders $24$ and $28$.

\subsection{The possible six-point supports}

\begin{proposition}\label{b:prop:lift-orbits}
If a feasible spectrum with six positive coefficients exists in order
$20$, $24$, or $28$, then, up to an automorphism, its positive support
is one of the following:
\[
\begin{array}{c|c}
n&\{0,n/2,\pm r,\pm s\}\\ \hline
20&(r,s)=(1,4),(1,6)\\
24&(r,s)=(1,4),(1,8),(2,3)\\
28&(r,s)=(1,4),(1,10).
\end{array}
\]
\end{proposition}

\begin{proof}
Half-period descent and Theorem~\ref{a:thm:small-table} show that the
image in $C_{10}$, $C_{12}$, or $C_{14}$ is a feasible five-point
support.  The unit orbits of generating supports with one even and one
odd pair are
\[
\begin{array}{c|c|c}
m&\text{orbits}&\text{remaining orbits}\\ \hline
10&(1,2),(1,4)&(1,4)\\
12&(1,2),(1,4),(2,3),(3,4)&(1,4),(2,3)\\
14&(1,2),(1,4),(1,6)&(1,4).
\end{array}
\]
The $(1,2)$ family is excluded by Lemma~\ref{b:lem:consecutive-pairs}.
For $(3,4)$ in $C_{12}$, put $a_0=A$.  Write the negative masses on
$\pm1,\pm2,\pm5,6$ as $x_1,x_2,x_5,x_6$ and the positive masses on
$\pm3,\pm4$ as $C,D$.  By Lemmas~\ref{a:lem:basic} and
\ref{a:lem:quotient}, the positive support after compression modulo $4$
generates $C_4$.  Hence $x_1+x_5<C$, whereas
\[
 (a*a)_2=x_1^2+x_5^2-2Ax_2-2C(x_1+x_5)-2D(x_2+x_6)\le0.
\]
The orbit $(1,6)$ in $C_{14}$ compresses modulo $7$ to at most three
positive coefficients, contrary to $\kappa_2(7)=5$.
A unit modulo the even integer $m$ also represents a unit modulo
$2m$.  We may therefore lift the listed representatives.  The inverse
pair $\{\pm r\}$ in $C_m$ has the two lifts $\{\pm r\}$ and
$\{\pm(m-r)\}$ in $C_{2m}$.  Above $(1,4)$, the four choices are
\[
 (1,4),\quad(1,m-4),\quad(m-1,4),\quad(m-1,m-4).
\]
Multiplication by $m-1$ identifies the last two with the first two.
Above $(2,3)$ in $C_{12}$, the four choices
$(2,3),(2,9),(10,3),(10,9)$ are the images of $(2,3)$ under
multiplication by $1,11,7,5$ modulo $24$, respectively.
These are exactly the lifts in the statement.
\end{proof}

\subsection{Order twenty-four}

\begin{theorem}\label{b:thm:order-24}
\[
 \kappa_2(24)=7.
\]
\end{theorem}

\begin{proof}
We first exclude the two supports
\[
 P_\varepsilon=\{0,12,\pm1,\pm\varepsilon\},
 \qquad \varepsilon\in\{4,8\}.
\]
Suppose that $a=p-q$ is feasible, with
\[
 p_0=A,\quad p_{12}=B,\quad p_{\pm1}=C,\quad
 p_{\pm\varepsilon}=D,\qquad q_{\pm k}=x_k\ge0.
\]
Put $\rho=12-\varepsilon$, $\xi=x_\rho$, and
\[
 y=x_2+x_{10},\quad u=x_3+x_9,\quad
 v=x_5+x_7,\quad z=x_6,
\]
\[
 \alpha=A+B,\qquad \gamma=C-x_{11},\qquad
 \delta=D-\xi.
\]
The half-period quotient on $C_{12}$ is
\[
 b=(\alpha,\gamma,-y,-u,\delta,-v,-2z,-v,
       \delta,-u,-y,\gamma),
\]
and its five possible positive positions are all positive because
$\kappa_2(12)=5$.  Its compression modulo $4$ is
\[
 (\alpha+2\delta,\ \gamma-u-v,\ -2y-2z,\ \gamma-u-v).
\]
By Lemma~\ref{a:lem:basic}, the inverse Fourier value at $0$ exceeds
the value at $1$, giving
\begin{equation}
 \gamma>2y+u+v+2z.
 \label{b:eq:c24-gamma}
\end{equation}
Compressing $b$ modulo $6$ gives
\[
 c=(\alpha-2z,\gamma-v,\delta-y,-2u,\delta-y,\gamma-v).
\]
By Lemma~\ref{a:lem:basic} and \eqref{b:eq:c24-gamma},
$\alpha-2z>\gamma-v>0$.  At the half-period,
\[
 \frac{(c*c)_3}{4}
 =(\gamma-v)(\delta-y)-u(\alpha-2z)>0.
\]
It follows that
\begin{equation}
 \delta>y+u,
 \label{b:eq:c24-delta}
\end{equation}
including when $u=0$.  Thus
\begin{equation}
 C>x_{11}+2y+u+v+2z,\qquad D>\xi+y+u.
 \label{b:eq:c24-CD}
\end{equation}

For
\[
 R_4=\pm\{6,7,9,10\},\qquad
 R_8=\pm\{3,5,6,10\},
\]
write $q_0=q\one_{C_{24}\setminus R_\varepsilon}$.
The eight coordinates
$\{0,1,\varepsilon,1+\varepsilon\}+\{0,12\}$ are distinct.
Their length-three differences are $\pm7,\pm9$ when
$\varepsilon=4$, and $\pm3,\pm5$ when $\varepsilon=8$.
All belong to $R_\varepsilon$, so Lemma~\ref{lem:cube-boundary} gives
\[
 \sum_{R_\varepsilon}(q_0*q_0-2p*q_0)>0.
\]

For $\varepsilon=4$, the remaining inverse-pair representatives
are $2,3,5,8,11$.  Put
\[
 u=x_2+x_3+x_5+x_{11},\qquad v=x_2+x_3,\qquad \xi=x_8.
\]
The quotient bounds give $C>u$ and $D>v$.
For any two signed representatives from $\{\pm2,\pm3,\pm5,\pm11\}$,
at most one of their sum and difference lies in $R_4$.
The pair $\{\pm8\}$ has no self-convolution into $R_4$.
Its cross contribution to $\frac14\sum_{R_4}q_0*q_0$ is at most
$2u\xi$.  Hence
\[
\begin{aligned}
 \frac14\sum_{R_4}q_0*q_0
 &\le\frac{u^2}{2}+2u\xi\\
 &\le u(u-v+2\xi)+v^2\\
 &\le C(u-v+2\xi)+Dv\le\frac12\sum_{R_4}p*q_0.
\end{aligned}
\]
The second inequality uses
$u(u-v)+v^2-u^2/2=((u-v)^2+v^2)/2\ge0$.
The final transport weights are at least $D$ at $2,3$, at least $C$
at $5,11$, and at least $2C$ at $8$.

For $\varepsilon=8$, $q_0$ has representatives $2,4,7,9,11$.
Put $u=x_2+x_7+x_9+x_{11}$ and $v=x_4$, so $C>u$.
The pair $\{\pm4\}$ has no self-convolution into $R_8$.
Between it and every other inverse pair, at most one of the
sum and difference lies in $R_8$.
In $\frac14\sum_{R_8}q_0*q_0$, the first group contributes at most
$u^2$, and the cross terms at most $uv$.  Every remaining representative has transport
weight at least $C$, so
\[
 \frac14\sum_{R_8}q_0*q_0\le u(u+v)
 \le C(u+v)\le\frac12\sum_{R_8}p*q_0.
\]
Both cases contradict the strict inequality above.

It remains to exclude
\[
 P=\{0,12,\pm2,\pm3\}.
\]
Write
\[
 p_0=A,\quad p_{12}=B,\quad p_{\pm2}=C,\quad p_{\pm3}=D,
\]
and $q_{\pm k}=x_k$ for
$k\in\{1,4,5,6,7,8,9,10,11\}$.  Put
\[
 y=x_1+x_{11},\quad u=x_4+x_8,\quad v=x_5+x_7,
 \quad z=x_6,\quad w=y+v,
\]
\[
 \alpha=A+B,\qquad\gamma=C-x_{10},\qquad\delta=D-x_9.
\]
The half-order quotient is
\[
 b=(\alpha,-y,\gamma,\delta,-u,-v,-2z,-v,-u,\delta,\gamma,-y).
\]
Its compression to $C_6$ is
\[
 (\alpha-2z,-w,\gamma-u,2\delta,\gamma-u,-w).
\]
Comparing its inverse Fourier value at $0$ with those at $2$ and $3$
gives, respectively,
\begin{equation}
 \gamma>u+w,\qquad\delta>w.
 \label{b:eq:c24-p23-quotient}
\end{equation}
Furthermore,
\[
 \frac{(b*b)_6}{2}
 =\delta^2+2yv-2\alpha z-2\gamma u>0.
\]
Since $w<\delta$, we obtain
\[
 2u^2\le2\gamma u<\delta^2+2yv
 \le\delta^2+w^2<2\delta^2,
\]
and hence $\delta>u$.  In particular,
\begin{equation}
 C>x_{10}+u+w,
 \qquad D>u\ge x_4/2.
 \label{b:eq:c24-p23-delta}
\end{equation}

For $R=C_{24}\setminus(P+P)=\pm\{7,8,11\}$, write
$q_0=q\one_{C_{24}\setminus R}$.
The eight coordinates $\{0,2,3,5\}+\{0,12\}$ are distinct,
and their length-three differences $\pm7,\pm11$ lie in $R$.
Lemma~\ref{lem:cube-boundary} gives
\[
 \sum_R(q_0*q_0-2p*q_0)>0.
\]
Only the representatives $1,4,5,6,9,10$ remain in $q_0$.
Put
\[
 Q=x_{10}+x_4+x_1+x_5,\qquad K=x_6+x_9+x_{10},
\]
so $C>Q$ by \eqref{b:eq:c24-p23-delta}.  Then
\[
\begin{aligned}
 \frac14\sum_Rq_0*q_0
 &=(x_1+x_5+x_{10})x_6+(x_1+x_4)x_9+x_1x_{10}
   +\frac{x_4^2}{2}\\
 &\le QK+\frac{x_4^2}{2}\le CK+Dx_4
 \le\frac12\sum_Rp*q_0.
\end{aligned}
\]
The last inequality uses the transport weight at least $C$ at
$6,9,10$ and at least $D$ at $4$.
This contradicts the preceding strict inequality.

Proposition~\ref{b:prop:lift-orbits} has now excluded every six-point
support, so $\kappa_2(24)\ge7$.  The upper bound follows from
Proposition~\ref{prop:rational-constructions}.
\end{proof}

\subsection{Order twenty-eight}

\begin{proposition}\label{b:prop:common-six-support}
For $n\in\{28,30\}$, the positive support
$\{0,n/2,\pm1,\pm4\}$ is infeasible.
\end{proposition}

\begin{proof}
Put $m=n/2$ and write $a=p-q$, where
\[
 p_0=A,\qquad p_m=B,\qquad p_{\pm1}=C,\qquad p_{\pm4}=D.
\]
Set $S=C+D$, $M=\sum q_k$, and $h=\mathcal F^{-1}a$.
Equation~\eqref{eq:four} with $(r,s)=(1,4)$ gives $A>S$.
The difference function $h(0)-h(x)$ has total mass $B+2C+2D-M$.
Its Fourier coefficient bounds at $m$ and $1$ give
\begin{equation}
 M<2S,\qquad M\le B+C+2D.
 \label{b:eq:common-six-mass}
\end{equation}
Strictness in the first inequality follows because the character
$(-1)^x$ takes both signs on the nonzero positions, where
$h(0)-h(x)>0$.

Let $R=C_n\setminus(P+P)$.  The boundary weights are the same
in both orders; up to inversion, they are
\begin{equation}
\begin{array}{c|ccccc}
k&2,3&5&8&m-4&m-1\\ \hline
W(k)&B+D&B+C+D&B+2C+D&2C+D&C+2D.
\end{array}
\label{b:eq:common-six-weights}
\end{equation}
Indeed,
\[
 P+P=\{0,m,\pm1,\pm2,\pm3,\pm4,\pm5,\pm8,m\pm1,m\pm4\},
\]
and the displayed positions are distinct in orders $28$ and $30$.
All remaining negative positions lie in $R$ and have $W\ge A>S$.
If $B\ge C$, every weight is at least $S$, contradicting
Lemma~\ref{lem:complement-pair} and $M<2S$.
Thus put $t=C-B>0$ and $H=B+D$, so that $S=H+t$ and
$M\le S+H$.

Let $E=\{\pm2,\pm3\}$.  One has $W\ge H$ on $E$ and
$W\ge S$ elsewhere.  For $k,l\in E$, at most one of
$k+l,k-l$ belongs to $R$: the only possible such values are
$\pm6$, arising when $k,l\in\{\pm3\}$.
If $M\le4H$, it follows for every $k,l\in\supp q$ that
\[
 W(k)+W(l)\ge
 \frac M2\bigl(\one_R(k+l)+\one_R(k-l)\bigr).
\]
For a pair not contained in $E$, use $W(k)+W(l)\ge S+H\ge M$.
The even form of Lemma~\ref{lem:complement-pair} therefore gives
\begin{equation}
 M>4H.
 \label{b:eq:common-six-interval}
\end{equation}

Finally put $\lambda=t/C\in(0,1)$ and
$\omega=\one_R+\lambda\one_{\{\pm3\}}$.
The added weight raises the transport weight at $\pm2$ by
$\lambda C=t$ and at $\pm3$ by $\lambda A>t$.  Hence
\[
 \ip{\omega}{p*q}\ge SM,\qquad
 \ip{\omega}{p*p}=4Dt,\qquad
 \ip{\omega}{q*q}\le M^2,
\]
where the last inequality uses $0\le\omega\le1$.
Strict positivity of $a*a$ gives
\[
 M(2S-M)<4Dt.
\]
But $M\le S+H$ gives $2S-M\ge S-H=t$, and
\eqref{b:eq:common-six-interval} therefore yields
\[
 M(2S-M)\ge Mt>4Ht\ge4Dt,
\]
a contradiction.
\end{proof}

\begin{theorem}\label{b:thm:order-28}
\[
 \kappa_2(28)=7.
\]
\end{theorem}

\begin{proof}
By Proposition~\ref{b:prop:lift-orbits}, only
\[
 P_4=\{0,14,\pm1,\pm4\},\qquad
 P_{10}=\{0,14,\pm1,\pm10\}
\]
can support a feasible spectrum with six positive coefficients.
Proposition~\ref{b:prop:common-six-support} excludes $P_4$.

For $P_{10}$, write $a=p-q$, with
\[
 p_0=A,\qquad p_{14}=B,\qquad p_{\pm1}=C,\qquad p_{\pm10}=D,
\]
and $q_{\pm k}=x_k\ge0$ for
$k\in\{2,3,4,5,6,7,8,9,11,12,13\}$.
Set $S=C+D$ and $M=\sum q_k$.  Lemma~\ref{lem:deficit} at
frequency $14$ gives $B\le B+2S-M$, so $M\le2S$.
The quotient modulo $7$ has at most five positive positions,
$0,\pm1,\pm3$.  Since $\kappa_2(7)=5$, its coefficient at $1$
is positive, giving
\[
 C>x_6+x_8+x_{13}.
\]

Let $R=C_{28}\setminus(P_{10}+P_{10})=\pm\{3,5,6,7,12\}$,
and write $q_0=q\one_{C_{28}\setminus R}$.
The eight coordinates $\{0,1,10,11\}+\{0,14\}$ are distinct,
and their length-three differences $\pm3,\pm5$ belong to $R$.
Lemma~\ref{lem:cube-boundary} therefore gives
\[
 \sum_R(q_0*q_0-2p*q_0)>0.
\]
Only the representatives $2,4,8,9,11,13$ remain in $q_0$.
Put
\[
 u=x_2+x_4+x_{11}+x_{13},\qquad N=u+x_8+x_9\le M/2\le S.
\]
For pairs from $\{\pm2,\pm4,\pm11,\pm13\}$, at most one of
the sum and difference lies in $R$.
Every other mixed coefficient in $\frac14\sum_Rq_0*q_0$ is at most
$1$, except for $x_8x_{13}$, whose coefficient is $2$.
The pair $\{\pm9\}$ has no self-convolution into $R$ and does not
interact there with $\{\pm8\}$.  Consequently
\[
\begin{aligned}
 \frac14\sum_Rq_0*q_0
 &\le\frac{u^2}{2}+u(x_8+x_9)+x_8x_{13}+\frac{x_8^2}{2}\\
 &\le Nu+(x_8+x_{13})x_8\le Su+Cx_8
 \le\frac12\sum_Rp*q_0.
\end{aligned}
\]
The last inequality uses transport weight at least $S$ at
$2,4,11,13$ and at least $C$ at $8$.

This contradicts the preceding strict inequality.
Thus $\kappa_2(28)\ge7$.

The upper bound follows from
Proposition~\ref{prop:rational-constructions}.
\end{proof}

\section{The order twenty-six}\label{b:sec:26}

The half-order quotient now has odd order.  By
Theorem~\ref{b:thm:half-period-descent} and $\kappa_2(13)=5$, a
six-positive spectrum on $C_{26}$ would compress to a five-positive
spectrum on $C_{13}$.  We first determine the possible support of
that quotient and then use the parity restriction in
Theorem~\ref{b:thm:three-three} to determine its lifts.

\subsection{The quotient of order thirteen}

\begin{theorem}\label{b:thm:c13-support}
If $a$ is feasible on $C_{13}$ and has five positive coefficients,
then, up to an automorphism of $C_{13}$,
\[
 \{k:a_k>0\}=\{0,\pm1,\pm5\}.
\]
Conversely, this support is feasible.
\end{theorem}

\begin{proof}
Up to multiplication by a unit, signs, and interchange, the possible
supports are
\[
 \{0,\pm1,\pm2\},\qquad
 \{0,\pm1,\pm3\},\qquad
 \{0,\pm1,\pm5\}.
\]
We exclude the first two.

Lemma~\ref{b:lem:consecutive-pairs} excludes the first support.

For $P=\{0,\pm1,\pm3\}$ write
\[
 a_0=A,\quad a_{\pm1}=C,\quad a_{\pm3}=D,\quad
 a_{\pm j}=-x_j\quad(j=2,4,5,6).
\]
At frequency $5$,
\begin{equation}
 (a*a)_5=2x_2x_6+x_4^2-2Ax_5
 -2C(x_4+x_6)-2D(x_2+x_5).
 \label{b:eq:c13-13-position-five}
\end{equation}
If $C\ge D$, \eqref{a:eq:strong-mass} gives
$x_2+x_4+x_5+x_6\le(C+2D)/2$.  Hence
$x_4^2\le2Cx_4$ and
\[
 x_2x_6\le Dx_2+Cx_6:
\]
the latter is immediate when $x_2\le C$, while for $x_2>C$ one has
$x_6<(2D-C)/2<D$.  Equation
\eqref{b:eq:c13-13-position-five} is therefore nonpositive.

It remains to treat $D>C$.  Multiplication by $3^{-1}$ changes the
support to $\{0,\pm1,\pm4\}$.  In these coordinates, write $L=D$ and
$c=C$ for the weights on $\pm1$ and $\pm4$, and put
$q_{\pm j}=x_j$ for $j\in\{2,3,5,6\}$ and
$M=2(x_2+x_3+x_5+x_6)$.  The complement is
$\{\pm6\}$; its weights at $2,3,5,6$ are
$c,c,L,A+L$.  Moreover,
\[
 \sum_{\{\pm6\}}(q*q)=4x_2x_5+2x_3^2\le\frac{M^2}{2}.
\]
Consequently positivity on $\{\pm6\}$ gives $M>4c$, whereas
\eqref{a:eq:strong-mass} gives $M\le L+2c$, so $L>2c$.
Take
\[
 w=\one_{\{\pm6\}}+\frac{L-c}{L}\one_{\{\pm3\}}.
\]
Then
\[
 \ip{w}{p*q}\ge LM,\qquad
 \ip{w}{p*p}=4c(L-c),\qquad
 \ip{w}{q*q}\le\frac23M^2.
\]
For the last estimate, expand the two contributing inverse pairs and use
\[
\begin{aligned}
 \ip{w}{q*q}
 &\le8x_2x_5+2x_3^2+4x_3x_6+2x_5^2\\
 &\le\frac83(x_2+x_5)^2+2(x_3+x_6)^2
 \le\frac23M^2.
\end{aligned}
\]
The middle inequality follows from
$8uv+2v^2\le\frac83(u+v)^2$ and
$2u^2+4uv\le2(u+v)^2$ for $u,v\ge0$.
Thus
\[
 0<4c(L-c)-2LM+\frac23M^2.
\]
Since $M\le L+2c$, $M>4c$, and $L>2c$, the last upper bound satisfies
\[
 4c(L-c)-2LM+\frac23M^2
 \le4(L-c)\left(c-\frac M3\right)<0,
\]
a contradiction.  This excludes the second orbit.

Finally, the last support is attained by
Proposition~\ref{prop:additive-constructions}.
\end{proof}

\subsection{Exclusion of six and attainment of seven}

\begin{proposition}\label{b:prop:c26-unique-lift}
If a feasible spectrum on $C_{26}$ has six positive coefficients, then,
up to an automorphism, its positive support is
\[
 \{0,13,\pm2,\pm3\}.
\]
\end{proposition}

\begin{proof}
Theorem~\ref{b:thm:three-three} and half-period descent give a feasible
five-positive support on $C_{13}$.  By
Theorem~\ref{b:thm:c13-support}, its image is equivalent to
$\{0,\pm1,\pm5\}$.  Every unit modulo $13$ has an odd
representative, so this change of coordinates lifts to $C_{26}$.
The opposite-parity choices above $(1,5)$ are $(1,8)$ and $(12,5)$.
Multiplication by $5$ identifies these two choices, and multiplication
by $3$ sends $(1,8)$ to $(3,2)$, up to signs.  This gives the displayed
support.
\end{proof}

\begin{theorem}\label{b:thm:order-26}
\[
 \kappa_2(26)=7.
\]
\end{theorem}

\begin{proof}
Suppose that a feasible six-positive spectrum exists.  By
Proposition~\ref{b:prop:c26-unique-lift}, write
\[
 p_0=A,\qquad p_{13}=B,\qquad p_{\pm2}=C,\qquad p_{\pm3}=D,
\]
and put $a=p-q$, where $q_{\pm k}=x_k\ge0$ for
$k\in\{1,4,5,6,7,8,9,10,11,12\}$.  Set
\[
 M=2\sum_kx_k,\qquad S=C+D,\qquad H=B+D.
\]
Lemma~\ref{lem:deficit} gives
\begin{equation}
 M\le2S,\qquad M\le B+C+2D.
 \label{b:eq:c26-input}
\end{equation}
The feasible quotient on $C_{13}$ has at most five positive positions;
since $\kappa_2(13)=5$, all five must be positive.  Its coefficient
at $\pm3$ is $D-x_{10}$, so $x_{10}<D$.

We show that $x_4\le2H$.
Put $Y=x_1+x_5+x_7+x_9+x_{11}$.
The odd block has positive total mass by
Lemma~\ref{lem:parity-single}, giving $Y<(B+2D)/2<H$.
If $x_4>2H$, then $M/2>2H$ and
\eqref{b:eq:c26-input} gives $C\ge M-2H+B>M-2H>M/2$.
Since $x_4+x_{10}\le M/2$ and $x_{12}\le M/2$,
we have $C-x_{10}>x_4$ and $C>x_{12}$.
At position $6$, the odd-negative products have total at most $Y^2$.
Using $A>C$ from Lemma~\ref{a:lem:basic}, we obtain
\[
\begin{aligned}
 0<(a*a)_6
 &\le D^2+x_{10}^2+Y^2-2(C-x_{10})x_4
          -2(C-x_{12})(x_6+x_8)\\
 &<3H^2-2x_4^2<0,
\end{aligned}
\]
a contradiction.

Let $R=C_{26}\setminus(P+P)=\pm\{7,8,9,12\}$ and put
$q_0=q\one_{C_{26}\setminus R}$.
The eight coordinates $\{0,2,3,5\}+\{0,13\}$ are distinct,
and their length-three differences $\pm8,\pm12$ lie in $R$.
Lemma~\ref{lem:cube-boundary} therefore gives
\[
 \sum_R(q_0*q_0-2p*q_0)>0.
\]
Only the representatives $1,4,5,6,10,11$ remain in $q_0$.
Put $v=x_5+x_6+x_{10}+x_{11}$ and $N=x_1+x_4+v\le S$.
In $\frac14\sum_Rq_0*q_0$, the only term involving solely
$x_1,x_4$ is $x_4^2/2$; all other square and mixed coefficients
are at most $1/2$ and $1$.  The transport weight into $R$ is $H$
at $4$ and at least $S$ at $5,6,10,11$.  Consequently
\[
 \frac14\sum_Rq_0*q_0
 \le\frac{x_4^2}{2}+Nv
 \le Hx_4+Sv
 \le\frac12\sum_Rp*q_0.
\]
This contradicts the preceding strict inequality.
Since Corollary~\ref{cor:base-bounds} gives $\kappa_2(26)\ge6$,
excluding six proves $\kappa_2(26)\ge7$.  The upper bound follows from
Proposition~\ref{prop:rational-constructions}.
\end{proof}

\section{The order thirty}\label{b:sec:30}

Half-period descent reduces a possible six-positive spectrum on
$C_{30}$ to a five-positive spectrum on $C_{15}$.
The support classification below identifies its possible lifts;
Proposition~\ref{b:prop:common-six-support} will then exclude them.

\subsection{The quotient of order fifteen}

The short relation in Theorem~\ref{a:thm:short-relation} reduces
the feasible five-point supports in $C_{15}$, up to an
automorphism, to
\begin{equation}
 (1,2),\quad(1,3),\quad(1,4),\quad(1,5),\quad(1,6),\quad(3,5).
 \label{b:eq:c15-orbits}
\end{equation}
Indeed, if both relation coefficients are nonzero, their absolute
values are coprime: the only other possibility is $(2,2)$, which
would identify the two inverse pairs.  The B\'ezout argument in
the proof of Proposition~\ref{a:prop:six-families} therefore gives
$(1,2),(1,3),(1,4)$ or $(2,3)$; multiplication by $8$ sends
the last pair to $(1,6)$ up to signs.
If one coefficient is zero, one frequency has order three or five.
When the other is a unit, normalization gives $(1,5)$, $(1,3)$
or $(1,6)$.  Otherwise the generating condition forces the two
frequencies to have orders three and five, giving $(3,5)$.

\begin{proposition}\label{b:prop:c15-four-exclusions}
The four support orbits $(1,3)$, $(1,5)$, $(1,6)$, and $(3,5)$ in
\eqref{b:eq:c15-orbits} are infeasible.
\end{proposition}

\begin{proof}
Write $a=p-q$, with $p_0=A$, and use $x_j=q_j=q_{-j}$ for the
negative inverse pairs.  Put $R=G\setminus(P+P)$ and
$q_0=q\one_{G\setminus R}$.  Lemma~\ref{a:lem:strong-mass} gives
$A>\sum q_k$, so Lemma~\ref{lem:boundary-reduction} requires
\begin{equation}
 \sum_Rq_0*q_0>2\sum_Rp*q_0.
 \label{eq:c15-boundary}
\end{equation}
The quotients modulo three and five give the opposite inequality
for three of the supports.

For $P=\{0,\pm1,\pm3\}$, put $p_{\pm1}=C$ and $p_{\pm3}=D$.
The coefficients at the two nonzero frequencies of the quotient on
$C_3$ are $C-x_2-x_4-x_5-x_7$.  Their positivity follows from
Lemma~\ref{lem:deficit}; hence $x_2+x_4<C$.
Here $R=\pm\{5,7\}$ and $q_0$ is supported on $\pm\{2,4,6\}$.
The weights into $R$ at $4,6$ are $C+D,2C$, respectively, so
\[
 \sum_Rq_0*q_0=2x_4^2+4(x_2+x_4)x_6
 \le2C(x_4+2x_6)\le\sum_Rp*q_0,
\]
contrary to \eqref{eq:c15-boundary}.

For $P=\{0,\pm1,\pm5\}$, put $p_{\pm1}=C$ and $p_{\pm5}=D$.
The quotient on $C_5$ has the form $(A+2D,H,-Q,-Q,H)$, where
$H=C-x_4-x_6$ and $Q=x_2+x_3+x_7$.  The deficit bound gives
$2Q\le H$, and thus $U=x_2+x_4+x_6\le C$.
Now $R=\pm\{3,7\}$ and $q_0$ is supported on $\pm\{2,4,6\}$.
Each of these positions enters $R$ with weight at least $C$; hence
\[
 \sum_Rq_0*q_0=4x_2x_6+2x_4^2+2x_6^2
 \le2U^2\le2CU\le\sum_Rp*q_0,
\]
again contradicting \eqref{eq:c15-boundary}.

For $P=\{0,\pm3,\pm5\}$, put $p_{\pm3}=C$ and $p_{\pm5}=D$.
Reduction modulo five gives
$2(x_1+x_4+x_6)+x_2+x_7\le C$ by the same deficit bound.
Here $R=\pm\{1,4\}$ and $q_0$ is supported on $\pm\{2,6,7\}$.
With $U=x_2+x_7$, the weights at $2,7$ are at least $C$, and
\[
 \sum_Rq_0*q_0=2x_2^2+2x_7^2+4x_6U
 \le2U(U+2x_6)\le2CU\le\sum_Rp*q_0.
\]
This also contradicts \eqref{eq:c15-boundary}.

For the last support $P=\{0,\pm1,\pm6\}$, write
$p_{\pm1}=C$, $p_{\pm6}=D$, and
$a_{\pm j}=-x_j$ for $j=2,3,4,5,7$.  Reduction modulo three and
five gives, respectively,
\[
 x_2+x_4+x_5+x_7<C,\qquad
 2(x_2+x_3+x_7)+x_4\le C+D.
\]
At position four,
\[
 (a*a)_4=x_2(x_2-2D)-2(C+D)x_5
 -2(C-x_7)x_3-2(A-x_7)x_4>0.
\]
Thus $x_2>2D$ and $2(C+D)x_5<x_2(x_2-2D)$.
Also $x_2+x_7\le(C+D)/2$, so $C>3D$.
At position three, using $A>C$ and $x_7<C$, we obtain
\[
 \begin{aligned}
 (a*a)_3
 &\le D^2+2x_5(x_2+x_7)-2Cx_2\\
 &<D^2+\frac{x_2(x_2-2D)}2-2Cx_2\\
 &<-\frac{7C+D}{4}\,x_2<0.
 \end{aligned}
\]
For the last step, use $2D<x_2\le(C+D)/2$, so that
$D^2<Dx_2/2$ and $x_2^2/2\le(C+D)x_2/4$.
This completes the four exclusions.
\end{proof}

\begin{theorem}\label{b:thm:c15-support}
Every feasible spectrum on $\Z/15\Z$ with five positive coefficients
has, up to a group automorphism, positive support
\[
 \{0,\pm1,\pm4\}.
\]
This support is attained.
\end{theorem}

\begin{proof}
Lemma~\ref{b:lem:consecutive-pairs} excludes $(1,2)$.  Together with
Proposition~\ref{b:prop:c15-four-exclusions}, the support list
\eqref{b:eq:c15-orbits} therefore leaves only $(1,4)$.  This support is
attained by the order-$15$ construction in
Proposition~\ref{prop:rational-constructions}, whose positive pairs $\pm2,\pm7$
become $\pm1,\pm4$ after multiplication by $8$ modulo $15$.
\end{proof}

\subsection{Exclusion of six and attainment of seven}

\begin{proposition}\label{b:prop:c30-unique-lift}
If a feasible spectrum on $\Z/30\Z$ has six positive coefficients,
then, up to an automorphism, its positive support is
\[
 \{0,15,\pm1,\pm4\}.
\]
\end{proposition}

\begin{proof}
The positive support has the form $\{0,15,\pm r,\pm s\}$.
Theorem~\ref{b:thm:three-three} shows that $r$ and $s$ have opposite
parity, while Theorem~\ref{b:thm:half-period-descent} and
$\kappa_2(15)=5$ show that its image modulo $15$ is a five-positive
support.  By Theorem~\ref{b:thm:c15-support}, that image is equivalent to
$\{0,\pm1,\pm4\}$.  Every unit modulo $15$ has an odd
representative and hence lifts to a unit modulo $30$.
The two inverse pairs above $\pm1$ and $\pm4$
have opposite parity only for the choices $(1,4)$ and $(14,11)$.
Multiplication by the unit $11$ modulo $30$ identifies these choices.
\end{proof}

\begin{theorem}\label{b:thm:c30-exact}
\[
 \kappa_2(30)=7.
\]
Moreover, there exists a feasible matrix with exactly seven positive
and twenty-three negative eigenvalues.
\end{theorem}

\begin{proof}
Proposition~\ref{b:prop:c30-unique-lift} reduces a possible six-positive
support to $\{0,15,\pm1,\pm4\}$, and
Proposition~\ref{b:prop:common-six-support} excludes it.  Since
Corollary~\ref{cor:base-bounds} gives $\kappa_2(30)\ge6$, we obtain
$\kappa_2(30)\ge7$.

The upper bound and the nonsingular spectrum with seven positive and
twenty-three negative coefficients are supplied by
Proposition~\ref{prop:rational-constructions}.
\end{proof}

\begin{proof}[Proof of the cyclic assertions in Theorem~\ref{thm:main}]
Corollary~\ref{cor:base-bounds} gives the first lower bound and the
odd-order bound.  The equality $\kappa_2(15)=5$ in
Theorem~\ref{a:thm:small-table} shows that the first threshold is sharp.
Propositions~\ref{prop:additive-constructions} and
\ref{prop:rational-constructions} give the matching upper bounds six
at $16,18,20,22$ and seven at $17,19,21$.

For even orders at least $32$, Corollary~\ref{b:cor:seven-even}
gives the bound seven.  The remaining even orders at least $24$ are
settled by Theorems~\ref{b:thm:order-24}, \ref{b:thm:order-26},
\ref{b:thm:order-28}, and \ref{b:thm:c30-exact}.
Together with the odd-order bound, this proves $\kappa_2(n)\ge7$
for all $n\ge23$, and $\kappa_2(22)=6$ makes that threshold sharp.

Theorem~\ref{a:thm:small-table} excludes the value six for
$5\le n\le15$, while $\kappa_2(n)\le n\le4$ for $n\le4$.
The odd orders $17,19,21$ have minimum seven, and all orders at
least $23$ have minimum at least seven.  Thus six occurs precisely
at the four asserted orders.  The nonsingular order-$30$ example
is part of Proposition~\ref{prop:rational-constructions}.
\end{proof}

\section{Noncyclic groups with at most five positive coefficients}
\label{sec:noncyclic-five}

We now use the short relation in Theorem~\ref{a:thm:short-relation}
to determine the noncyclic groups with minimum smaller than six.

\begin{lemma}\label{c:lem:rank-two-reduction}
Let $G=C_m\times C_n$, where $1<m\mid n$.  If $G$ has a
feasible spectrum with at most five positive coefficients, then
\[
 G\in\{C_2^2,C_2\times C_4,C_2\times C_6,C_2\times C_8,
 C_3^2,C_4^2,C_5^2\}.
\]
If the positive support consists of zero and two inverse pairs,
then $G$ is one of $C_2\times C_4,C_2\times C_8,C_3^2,C_4^2,C_5^2$.
\end{lemma}

\begin{proof}
Inversion symmetry and the generating condition leave three types:
zero and involutions; zero, one inverse pair and at most two
involutions; or $P=\{0,\pm r,\pm s\}$.
The first type gives $G=C_2^2$.  In the second type, $2G$ is
generated by $2r$, so $m=2$.  Projection onto $C_n$ maps the
involutions into $\{0,n/2\}$ and gives at most four positive
positions.  Theorem~\ref{a:thm:no-four-even} gives $n\le8$.

For the third type, $r,s$ generate $G$.  Reduction modulo $mG$
gives a surjection $C_m^2\to G/mG\cong C_m^2$, hence an
isomorphism.  Thus every relation $\alpha r+\beta s=0$ has
$m\mid\alpha,\beta$.  Theorem~\ref{a:thm:short-relation}
supplies such a nonzero relation of length at most five, so $m\le5$.
If $3\le m\le5$, the relation is $mr=0$, after exchanging the
generators and changing signs.  The same divisibility shows
$\operatorname{ord}(r)=m$.  The quotient by $\langle r\rangle$
is cyclic of order $n$, with at most three positive positions.
Lemma~\ref{a:lem:three-positive} gives $n\le5$, hence $n=m$.

If $m=2$, the relation is $4r=0$ or $2(r\pm s)=0$; the
possibility $2r=0$ is excluded by the distinctness of the five
support elements.  In the first case $r$ has order four.  The
cyclic quotient by $\langle r\rangle$ has order $n/2\le5$,
while $4\mid n$, so $n=4$ or $8$.  In the second case
$t=r\pm s$ is nonzero, since otherwise $G$ would be cyclic.
The quotient by this involution is cyclic of order $n$ and has
at most three positive positions.  Hence $n\le5$, and $n=4$.
\end{proof}

The index-two lemma and the complement estimate now exclude the remaining groups.

\begin{proposition}\label{c:prop:local-five}
If $G$ is a noncyclic finite abelian $2$-group of order at least
sixteen, or if $G=C_5^2$, then $\mu_2(G)\ge6$.
\end{proposition}

\begin{proof}
First suppose that $G$ is a noncyclic $2$-group with $|G|\ge16$,
and let $P$ be an inversion-invariant generating set containing zero,
with $|P|\le5$.  We show that some homomorphism
$\chi:G\to C_2$ has exactly one point of $P$ in one of its fibres.

If every nonzero element of $P$ is an involution, then $G=C_2^4$ and
$P\setminus\{0\}$ is a basis; send all four basis vectors to one.
If $P$ contains two nontrivial inverse pairs, their representatives
form a basis of $G/2G$, since they generate the noncyclic group $G$.
Again send both basis vectors to one.  In both cases the zero fibre
meets $P$ only at zero.

Otherwise $P$ consists of zero, one inverse pair $\pm r$, and at
most two nonzero involutions.  If $G/2G$ has dimension three, these
three representatives form a basis and may all be sent to one.
If its dimension is two, then $2G=\langle2r\rangle$ is cyclic,
so $G=C_2\times C_{2^a}$ for some $a\ge3$.  The nonzero involutions in $P$ have,
modulo $2G$, either zero image or the same nonzero image $u$;
at least one has image $u$, and $u,\bar r$ form a basis.
If none has zero image, send both $u$ and $\bar r$ to one.
If one has zero image, send $u$ to one and $\bar r$ to zero;
then the fibre over one contains exactly one point of $P$.
Thus the asserted homomorphism exists in every case.

If $P$ were feasible, Lemma~\ref{lem:parity-single} would now give
$|P|\ge |G|/2+1\ge9$, a contradiction.

For $G=C_5^2$, a generating support of at most five points must
be $P=\{0,\pm r,\pm s\}$ with $r,s$ a basis.  Write
$p_0=A$, $p_{\pm r}=C$, $p_{\pm s}=D$, and assume $C\ge D$.
Put $S=C+D$.  Lemma~\ref{a:lem:strong-mass} gives
$M\le S+D$ and $A>S$.
For $R=G\setminus(P+P)$, translation by the positive modes gives
\[
 W(\pm2r)=2D,\qquad W(\pm2s)=2C,\qquad
 W(\pm r\pm s)\ge S,
\]
while $W\ge A>S$ on $R$.  With $E=\{\pm2r\}$, one has
$E+E=\{0,\pm r\}\subseteq P$.
Lemma~\ref{lem:complement-pair}, with $(u,v)=(S,D)$,
excludes this support.
\end{proof}

\begin{theorem}\label{c:thm:below-six}
For a noncyclic finite abelian group $G$, one has $\mu_2(G)<6$
if and only if
\[
 G\in\{C_2^2,C_2^3,C_2\times C_4,C_2\times C_6,C_3^2\}.
\]
The value is three for $C_2^2$ and five for the other four groups.
\end{theorem}

\begin{proof}
The following inversion-invariant sets satisfy $P+P=G$:
\[
\begin{array}{c|c}
G&P\\ \hline
C_2^2&\{0,(1,0),(0,1)\}\\
C_2^3&\{0,e_1,e_2,e_3,e_1+e_2+e_3\}\\
C_2\times C_4&\{0,\pm(0,1),(1,0),(1,2)\}\\
C_2\times C_6&\{0,\pm(0,1),(1,0),(1,3)\}\\
C_3^2&\{0,\pm(1,0),\pm(0,1)\}.
\end{array}
\]
They give the stated upper bounds by Lemma~\ref{a:lem:additive-basis}.
Generation gives the lower bounds for $C_2^2$ and $C_3^2$.
A generating support of at most four points on $C_2^3$ is zero
and a basis.  The index-two subgroup defined by sending all three
basis vectors to $1\in C_2$ contains only the positive zero mode.
Lemma~\ref{lem:parity-single} would require at least five positive
coefficients.  Thus $\mu_2(C_2^3)=5$.

For $C_2\times C_4$ and $C_2\times C_6$, a generating support
with at most four points must be $\{0,t,\pm r\}$, with $t$ an
involution.  Noncyclicity implies $t\notin\langle r\rangle$,
and this cyclic subgroup has index two and order at least four.
Its other coset contains just one positive coefficient.  Again
Lemma~\ref{lem:parity-single} gives a contradiction.  The rank-two
classification now follows from
Lemma~\ref{c:lem:rank-two-reduction} and
Proposition~\ref{c:prop:local-five}.

Let $d(G)\ge3$, where $d(G)$ is the least number of generators,
and suppose that $|P|\le5$.  A support consisting of involutions
gives $G=C_2^d$, with $3\le d\le4$.  The case $d=4$ would
have $P$ equal to zero and a basis; quotienting by a basis vector
would give four positive positions on $C_2^3$, which was just
excluded.  Two inverse pairs generate at most a rank-two group.
In the remaining case $P=\{0,u,v,\pm r\}$, with $u,v$
involutions, $d(G)=3$.  The quotient $G/\langle u\rangle$
has rank two and at most four positive positions.  The rank-two
classification forces this quotient to be $C_2^2$.  Thus $G$
has order eight and rank three, so $G=C_2^3$, contradicting
the presence of a non-involution $r$.
\end{proof}

\begin{corollary}\label{c:cor:rank-growth}
For every prime $p$ and $r\ge2$,
\[
 \mu_2(C_p^r)\ge
 \begin{cases}
 \mu_2(C_2^{r-1})+1,&p=2,\\
 \mu_2(C_p^{r-1})+2,&p\text{ odd}.
 \end{cases}
\]
In particular $\mu_2(C_2^r)\ge r+2$ for $r\ge3$, and
$\mu_2(C_3^r)\ge2r+1$ for $r\ge2$.
\end{corollary}

\begin{proof}
Choose a nonzero positive frequency $e$ and descend to
$C_p^r/\langle e\rangle$.  The images of $0,e$ coincide for
$p=2$, and those of $0,e,-e$ coincide for odd $p$.
Lemma~\ref{a:lem:quotient} gives the recurrences.  Their initial
values are supplied by Theorem~\ref{c:thm:below-six}.
\end{proof}

\section{Six positive coefficients on noncyclic groups}
\label{sec:noncyclic-six}

Six positive coefficients force a positive nonzero involution.
Quotienting by that involution reduces the question to the
classification just proved.

\begin{proposition}\label{c:prop:six-reduction}
If $G$ is noncyclic and $\mu_2(G)=6$, then either $|G|=16$ or
\[
 G\in\{C_3\times C_6,C_2\times C_{10},C_2\times C_{12},
          C_2\times C_{14},C_2^2\times C_6\}.
\]
\end{proposition}

\begin{proof}
Let $P$ be a positive support of size six.  Since $0\in P$
and all nontrivial inversion orbits have size one or two, $P$
contains an involution $t$.  For $H=G/\langle t\rangle$,
Lemma~\ref{a:lem:quotient} gives $\mu_2(H)\le5$.
By the cyclic results and Theorem~\ref{c:thm:below-six}, $H$ is
cyclic of order at most fifteen, or is one of the five noncyclic
groups in that theorem.

The odd-primary factors of $G$ and $H$ agree.  For the
two-primary factors, the invariant-factor classification gives
the following list, after discarding cyclic groups and groups
whose minimum is already smaller than six:
\[
\begin{array}{c|l}
H&\text{remaining possibilities for }G\\ \hline
C_m,\ m\le15&C_2\times C_8,\ C_2\times C_{10},\
 C_2\times C_{12},\ C_2\times C_{14}\\
C_2^2&\text{none}\\
C_2^3&C_2^4,\ C_2^2\times C_4\\
C_2\times C_4&C_2\times C_8,\ C_4^2,\ C_2^2\times C_4\\
C_2\times C_6&C_2\times C_{12},\ C_2^2\times C_6\\
C_3^2&C_3\times C_6.
\end{array}
\]
For the first row, a group with a cyclic quotient by an
order-two subgroup is cyclic or $C_2\times C_m$ with $m$ even.
The next four rows follow by listing the abelian $2$-groups of
order eight or sixteen and their quotients by order-two subgroups.
The last row gives $G\cong C_2\times C_3^2\cong C_3\times C_6$.
The table gives the assertion.
\end{proof}

\begin{lemma}\label{c:lem:outer-involution}
Let $P=\{0,t,\pm r,\pm s\}$ be a six-point positive support
on a finite abelian group $G$, where $t$ is an involution.
If $t\notin2G$ and $\dim_{\mathbb F_2}(G/2G)\ge2$, then
$|G|\le16$.
\end{lemma}

\begin{proof}
Suppose that $|G|>16$.  By Lemma~\ref{lem:parity-single}, no
index-two subgroup can split $P$ into blocks of sizes one and
five.  Put $V=G/2G$, and write bars for images in $V$.
For every linear functional $\lambda:V\to\mathbb F_2$ with
$\lambda(\bar t)=1$, the odd block contains $t$ and either
zero, one or two of the inverse pairs.  It must therefore
contain exactly one pair.  Thus
\[
 \lambda(\bar r+\bar s)=1=\lambda(\bar t).
\]
These functionals span $V^*$, so $\bar r+\bar s=\bar t$.
Since $P$ generates $G$, the two vectors $\bar r,\bar s$
span $V$.  The dimension assumption gives $\dim V=2$.

Choose the nonzero functional $\lambda$ with
$\lambda(\bar t)=0$.  Then
$\lambda(\bar r)=\lambda(\bar s)=1$, since these vectors
generate $V$.  This functional descends to
$H=G/\langle t\rangle$.  The quotient spectrum has at most
five positive coefficients, and its even block has only the
positive zero coefficient.  Lemma~\ref{lem:parity-single}
forces at least $|H|/2+1=|G|/4+1>5$ positive coefficients,
a contradiction.
\end{proof}

\begin{proposition}\label{c:prop:even-noncyclic}
Each of $C_2\times C_{10},C_2\times C_{12},C_2\times C_{14}$
and $C_2^2\times C_6$ has minimum at least seven.
\end{proposition}

\begin{proof}
Their minima are at least six by Theorem~\ref{c:thm:below-six}.
Suppose that a positive support has size six.

First let $G=C_2\times C_n$, with $n=10,12,14$, and put
$u=(1,0)$, $v=(0,n/2)$.  There are three nonzero involutions.
If all three lie in $P$, then descent by $\langle v\rangle$
identifies $0,v$ and $u,u+v$.  It gives at most four positive
positions on, respectively, $C_{10}$, $C_2\times C_6$ or
$C_{14}$.  Each has minimum five, a contradiction.
Thus $P=\{0,t,\pm r,\pm s\}$ with just one positive involution.
For $n=10,14$ every involution lies outside $2G$, so
Lemma~\ref{c:lem:outer-involution} applies.  For $n=12$ it
forces $t=v$, the only involution in $2G$.  Descent by
$\langle v\rangle$ gives a feasible spectrum on $C_2\times C_6$
with at most five positive positions.  Its minimum is five,
so the projected support is exactly zero and two distinct
inverse pairs.  This contradicts
Lemma~\ref{c:lem:rank-two-reduction}.

Now take $G=C_2^2\times C_6\cong C_2^3\times C_3$.
A support made of involutions cannot generate $G$, and the
one-involution case is excluded by
Lemma~\ref{c:lem:outer-involution}.  Hence
\[
 P=\{0,t_1,t_2,t_3,\pm r\},
\]
where the $t_i$ are distinct nonzero involutions.  If they are
dependent in the two-torsion subgroup, then
$t_1+t_2+t_3=0$.  Quotienting by $\langle t_1\rangle$
gives at most four positive positions on $C_2\times C_6$,
a contradiction.  Otherwise the $t_i$ are a basis of $C_2^3$.
Let $w$ be the $C_2^3$ component of $r$.  The functional
which is one on $t_i$ and zero on the other two gives one
positive odd coefficient unless it is also one on $w$.
Lemma~\ref{lem:parity-single} therefore forces
$w=t_1+t_2+t_3$.  The functional which is one on all three
$t_i$ then gives five positive odd coefficients and only the
positive zero coefficient in the even block.  The same lemma
again gives a contradiction.
\end{proof}

\begin{proposition}\label{c:prop:c3c6}
One has $\mu_2(C_3\times C_6)\ge7$.
\end{proposition}

\begin{proof}
Write $G=C_2\times H$, where $H=C_3^2$, and suppose that
$|P|=6$.  The unique involution must belong to $P$.
Lemma~\ref{lem:parity-single} forces three positive coefficients
in each $H$-coset.  Projection onto $H$ and the generating
condition show that the two inverse pairs project to a basis
$r,s$ of $H$.  On writing $a=(e,o)$ for the two cosets,
the coefficients are therefore of the following form:
\[
\begin{array}{c|ccccc}
 &0&\pm s&\pm r&\pm(r+s)&\pm(r-s)\\ \hline
 e&A&C&-u&-v&-w\\
 o&B&-x&D&-y&-z,
\end{array}
\]
where $A,B,C,D>0$ and $u,v,w,x,y,z\ge0$.
The positive block sums in Lemma~\ref{lem:parity-single} give
\[
 2(u+v+w)<A+2C,\qquad 2(x+y+z)<B+2D.
\]
Put $V=v+w$, $Y=y+z$ and $R=\{\pm r\pm s\}\subset H$.
The even coset of $a*a$ is $e*e+o*o$, and direct summation yields
\begin{align*}
 \frac12\sum_{k\in R}(a*a)_{(0,k)}
 ={}&v^2+w^2+2uV+y^2+z^2+2xY\\
 &-4Cu-2(A+C)V-4Dx-2(B+D)Y.
\end{align*}
But
\[
 v^2+w^2+2uV\le2(u+V)V\le(A+2C)V,
\]
and similarly $y^2+z^2+2xY\le(B+2D)Y$.
The displayed convolution sum is consequently at most
$-AV-BY-4Cu-4Dx\le0$, contrary to feasibility.
Theorem~\ref{c:thm:below-six} excludes smaller minima.
\end{proof}

\begin{proposition}\label{c:prop:order-sixteen}
Every abelian group of order sixteen has minimum six.
\end{proposition}

\begin{proof}
The lower bound follows from the cyclic results and
Theorem~\ref{c:thm:below-six}.  The cyclic upper bound is
already given by Proposition~\ref{prop:additive-constructions}.
For three of the noncyclic groups, use the following six-point
additive bases:
\[
\begin{array}{c|l}
G&P\\ \hline
C_4^2&\{0,\pm(1,0),\pm(0,1),(2,2)\}\\
C_2^2\times C_4&\{0,(1,0,0),(0,1,0),(1,1,2),\pm(0,0,1)\}\\
C_2^4&\{0,e_1,e_2,e_3,e_4,e_1+e_2+e_3+e_4\}.
\end{array}
\]
In the first row, the sums give the coordinate axes, the four
points $(\pm1,\pm1)$, and $(2,2)$ together with its four
neighbors.  In the second, they cover all four $C_4$-fibers.
In the third, they give all vectors of weights zero through
four.  Thus $P+P=G$ in each case, and
Lemma~\ref{a:lem:additive-basis} applies.

For $C_2\times C_8$, take the integer spectrum whose two rows are
\[
 a=\begin{pmatrix}
 100&26&-5&0&25&0&-5&26\\
 0&26&-5&0&0&0&-5&26
 \end{pmatrix}.
\]
It has exactly six positive coefficients and is inversion-invariant.
For the character $(\varepsilon,j)\in C_2\times C_8$, the
unnormalized inverse transform is
\[
 16h(\varepsilon,j)=
 \begin{cases}
 100+25(-1)^j+104\cos(\pi j/4)-20\cos(\pi j/2),&\varepsilon=0,\\
 100+25(-1)^j,&\varepsilon=1.
 \end{cases}
\]
For $\varepsilon=0$, the minimum is one: the only other potentially
small value is $75-52\sqrt2>1$, because $\sqrt2<37/26$.
For $\varepsilon=1$, the value is at least $75$.  Direct convolution gives
\[
 a*a=\begin{pmatrix}
 13429&4680&102&780&5100&780&102&4680\\
 2804&4680&102&780&100&780&102&4680
 \end{pmatrix}>0.
\]
This proves feasibility and completes the five group types of
order sixteen.
\end{proof}

\begin{proof}[Proof of Theorem~\ref{thm:main}]
The cyclic assertions were proved above.
Theorem~\ref{c:thm:below-six} gives the stated noncyclic minima
below six.  For the value six, Proposition~\ref{c:prop:six-reduction}
leaves order sixteen and five further noncyclic groups.
Propositions~\ref{c:prop:even-noncyclic} and \ref{c:prop:c3c6}
exclude the latter, while Proposition~\ref{c:prop:order-sixteen}
supplies the upper bounds for the former.
These classifications imply both lower thresholds, sharp by the
examples on $C_{15}$ and $C_{22}$.
\end{proof}

\end{document}